\documentclass[12pt]{amsart}
\usepackage{amsfonts,amssymb,amsmath,amsthm,amscd,graphicx,float,color}
 \usepackage{mathrsfs}

\newtheorem{theorem}{Theorem}[section]
\newtheorem{THM}{Theorem}
\newtheorem{lemma}[theorem]{Lemma}
\newtheorem{prop}[theorem]{Proposition}

\theoremstyle{definition}
\newtheorem{definition}[theorem]{Definition}

\theoremstyle{remark}
\newtheorem{remark}[theorem]{Remark}

\numberwithin{equation}{section}

\newcommand{\lcr}{\raisebox{-5pt}{\mbox{}\hspace{1pt}
                 \includegraphics{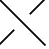}\hspace{1pt}\mbox{}}}
\newcommand{\ift}{\raisebox{-5pt}{\mbox{}\hspace{1pt}
                 \includegraphics{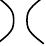}\hspace{1pt}\mbox{}}}
\newcommand{\zer}{\raisebox{-5pt}{\mbox{}\hspace{1pt}
                 \includegraphics{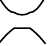}\hspace{1pt}\mbox{}}}

\newcommand{\no}[1]{}

\newcommand{\BC}{\mathbb C}
\newcommand{\BZ}{\mathbb Z}
\newcommand{\BN}{\mathbb N}
\newcommand{\bk}{\mathbf k}
\newcommand{\cB}{\mathcal B}
\newcommand{\hF}{\hat F}
\newcommand{\KeF}{K_\epsilon(F)}

\newcommand{\KF}{K_\zeta(F)}
\newcommand{\ZF}{Z_\zeta(F)}

\newcommand{\cD}{\mathcal{D}}

\newcommand{\br}{{\mathbf r}}
\newcommand{\pr}{\mathrm{pr}}

\newcommand{\Ir}{\mathrm{Irrep}}

\newcommand{\Hom}{\mathrm{Hom}}
\newcommand{\HA}{\mathrm{Hom}_{k-\mathrm{Alg}}}
\newcommand{\MS}{\mathrm{Max\ Spec}}
\newcommand{\fm}{\mathfrak m}
\newcommand{\bZR}{\overline{Z(R)}}

\newcommand{\ord}{\mathrm{ord}}

\newcommand{\cA}{\mathcal A}
\newcommand{\gr}{\mathrm{gr}}
\newcommand{\Gr}{\mathrm Gr}
 \newcommand{\I}{\mathrm{Ind}}
 \newcommand{\cP}{\mathcal P}
 \newcommand{\bn}{\mathbf n}
 \newcommand{\bt}{ {\mathbf t}}
 \newcommand{\cC}{\mathcal C}
\newcommand{\cV}{\mathcal V}

 \newcommand{\IT}{\mathrm{Ind}^\Delta}
 \newcommand{\Bt}{\cB^{\Delta}}

\newcommand{\di}{\boldsymbol{\delta}_i}
\newcommand{\al}{\alpha}

\def\Irrep{\mathrm{Irrep}}
\def\be{ \begin{equation}}
\def\ee { \end{equation}}
\def\onto{\twoheadrightarrow}

\def\CXe{\BC\la X^\Delta \ra_\epsilon}
\def\CXz{\BC\la X^\Delta \ra_\zeta}

\def\BR{{\mathbb R}}
\def\bx{\mathbf x}

 \def\MF{\mathcal{MF}}
 \def\PMF{\mathcal {PMF}}
 \def\Rzero{ \BR_+ }

  \def\by {\mathbf y}
  \def\pc{\mathrm{pc}}

\newcommand\FIGc[3]{\begin{figure}[htpb]
    \includegraphics[height=#3]{#1.pdf}
    \caption{#2}
    \label{fig:#1}
    \end{figure}}

\author{Charles Frohman}
\email{charles-frohman@uiowa.edu}
\author{Joanna Kania-Bartoszynska}
\email{jkaniaba@nsf.gov}
\author{ Thang L\^{e}}
\email{letu@math.gatech.edu}

\title{Unicity for Representations of the Kauffman bracket Skein Algebra}

\begin{document}
\begin{abstract} This paper resolves the unicity conjecture of Bonahon and Wong for the Kauffman bracket skein algebras of all oriented finite type surfaces at all roots of unity. The proof is a consequence of a general unicity theorem that says that the irreducible representations of a prime affine $k$-algebra over an algebraically closed field $k$, that is finitely generated as a module over its center, are generically classified by their central characters.  The center of the Kauffman bracket skein algebra of any orientable surface  at any root of unity is characterized,  and it is proved that the skein algebra is finitely generated as a module over its center. It is shown that for any orientable surface the center of the skein algebra at any root of unity is the coordinate ring of an  affine algebraic variety.
\end{abstract}
\maketitle

\section{Introduction}

Quantum hyperbolic geometry \cite{BB, K}  has the goal of building a quantum field theory that assigns numerical invariants to a three manifold $M$ equipped with a trace equivalence class of representations $\rho:\pi_1(M)\rightarrow PSL_2\mathbb{C}$, and perhaps with some homological data associated to cusps.
One approach to this is via quantum Teichm\"{u}ller theory \cite{K, BL, CF}. Deep work of Bonahon and Wong \cite{BW0} shows that the mechanism of quantum Teichm\"{u}ller theory is essentially describing the Kauffman bracket skein algebra of a finite type surface at roots of unity.  The next step is to understand the irreducible representations of those algebras  \cite{BW1,BW2,BW3}.  Bonahon and Wong associate to each irreducible representation of the skein algebra at roots of unity (with a mild restriction on the order) a {\em classical shadow}, which is determined by its 
central character.  They conjecture that there is a generic family of  classical shadows for which there is a unique irreducible representation of the skein algebra realizing that classical shadow \cite{BW3}. We call this the {\em unicity conjecture}.

In this paper we resolve the unicity conjecture for irreducible representations of the Kauffman bracket skein algebras of all finite type surfaces at all roots of unity. The proof proceeds by studying the irreducible representations of a localization of the skein algebra that is Azumaya. Azumaya algebras are central to the study of noncommutative algebraic geometry \cite{O}. The irreducible representations of an Azumaya algebra correspond exactly to the maximal ideals of its center.
We actually prove a more general result, Theorem \ref{THM.unicity} below. Recall that if $R$ is a $k$-algebra, where $k$ is an algebraically closed field, then every finite-dimensional representation $\rho$ of $R$ gives rise to a central character $\chi_\rho$, which is a $k$-algebra homomorphism from the center of $R$ to $k$.  Denote the center of $R$ by $Z(R)$. Let $\Irrep(R)$ be the set of all equivalence classes of finite-dimensional representations of $R$, and $\HA(Z(R), k)$ be the set of all $k$-algebra homomorphisms from $Z(R)$ to $k$, then we have the central character map
\be 
\label{eq.char}
\chi: \Irrep(R) \to \HA(Z(R), k), \quad \rho \to \chi_\rho.
\ee 

\def\MaxSpec{\mathrm{MaxSpec}}
\begin{THM}  [Unicity Theorem]
\label{THM.unicity}
Let $R$ be a prime  affine  
$k$-algebra, where $k$ is an algebraically closed field and $R$ is generated  as a module over its center $Z(R)$ by a finite set of $r$ elements.

(a) The central character map $\chi:\Irrep(R) \to \HA(Z(R),k)$ is surjective.

(b) There exists a  Zariski open and dense subset $V$ of $\ \HA(Z(R),k)$ such that if $\tau\in V$  then $\chi^{-1}(\tau)$ has one element, i.e. $\tau$ is the central character of a unique (up to equivalence) irreducible representation $\rho_\tau$. Moreover all  representations $\rho_\tau$ with $\tau\in V$ have the same dimension $N$, equal to the square root of the rank of $R$ over $Z(R)$. On the other hand if $\tau$ is not in $V$ then  it is the central character of at most $r$ non-equivalent irreducible representations, and each has dimension $\le N$.
\end{THM}

Some explanations are in order here, for details see Section \ref{unicitytheorem}. Since $R$ is prime, its center $Z(R)$ is a commutative integral domain. The rank of $R$ over $Z(R)$ is defined to be the dimension of the $Q$-vector space $R \otimes_{Z(R)} Q$, where $Q$ is the field of fractions of $Z(R)$. 
By  the Artin-Tate lemma (Lemma \ref{artintate}), $Z(R)$ is an affine commutative $k$-algebra. Hence by the weak Nullstellensatz \cite{AM}, the set $\HA(Z(R),k)$ can be identified with  $\MaxSpec(Z(R))$, the set of all maximal ideals of $Z(R)$, which is an affine algebraic set (over $k$) whose ring of regular functions is $Z(R)$. The Zariski topology of $\Hom_{k-Alg}(Z(R),k)$ in Theorem \ref{THM.unicity} is the Zariski topology of $\MaxSpec(Z(R))$. 
Since  $Z(R)$ is a domain, $\MaxSpec(Z(R))$ is an irreducible affine algebraic set. 

Theorem \ref{THM.unicity}, which is proved in Section \ref{unicitytheorem}, is an easy extension of  results from the study of polynomial identity rings. For instance,  it follows from the theorems in  \cite{GB}  used to analyze the representation theory of quantum groups at roots of unity. We have not found it stated explicitly in the literature, and it is useful for studying the representations of algebras underlying quantum invariants so we take a section to describe the concepts behind this theorem and sketch its proof.

\def\cX{\mathscr X}
\def\cY{\mathscr Y}
For an oriented finite type surface $F$ and a nonzero complex number  $\zeta$, let $\KF$ be the Kauffman bracket skein algebra of $F$ at $\zeta$,  and $\ZF$ be its center, see Section \ref{review}. When $\zeta$ is a root of 1 whose order is not a multiple of 4, Bonahon and Wong \cite{BW1} associate to each irreducible representation of $\KF$ an object called the {\bf classical shadow}, which is roughly the following.
Suppose first $F$ is a closed surface and the order of $\zeta$ is 2 mod 4. Bonahon and Wong   constructed an embedding of $\BC[\cX]$, the coordinate ring of the $SL_2(\BC)$-character variety $\cX$ of $F$, into the center $Z_\zeta(F)$ of the skein algebra $K_\zeta(F)$. For every irreducible representation of $\KF$, the restriction of its central character on $\BC[\cX]$ determines a point of $\cX$, i.e. an $SL_2(\BC)$-character, called the  classical shadow of the representation. When $F$ has punctures, the classical shadow can be defined similarly, and consists of  an $SL_2(\BC)$-character and other data coming from the punctures. We also extend the definition of the classical shadow to the case when $\zeta$ is an arbitrary root of 1. The set of possible classical shadows form an affine algebraic variety $\cY$, called the variety of classical shadows, 
which is closely related to the character variety $\cX$, see  Section \ref{finite}. For example, if $F$ is closed and $\zeta$ is a root of 1 of order equal to $2 \mod 4$, then $\cY=\cX$, which does not depend on $\zeta$. 

We will use Theorem \ref{THM.unicity} to prove the unicity conjecture of Bonahon and Wong:

\begin{THM}[Unicity Theorem for skein algebras]  \label{THM.2}
Let 
$F$ be an oriented finite type surface and $\zeta$ be 
a 
 root of unity.  There  
is a Zariski open dense subset $U$ of the variety of classical shadows $\cY$ 
 such that  each point of $U$ is the classical shadow of a unique (up to equivalence) irreducible representation of $\KF$.   All irreducible representations with classical shadows in $U$ have the same dimension $N$ which is  equal to the  square root of the rank of           
$\KF$ over $\ZF$. If a classical shadow is not in $U$, then it has at most $r$ non-equivalent irreducible representations, and each has dimension $\le N$. Here $r$ is a constant depending on the surface $F$ and the root $\zeta$.
\end{THM}

Let us look again at the case when $F$ is closed and the order of $\zeta$ is 2 mod 4. If the ring $\BC[\cX]$ of $SL_2(\BC)$-characters is strictly less than the center $\ZF$, then it is easy to show that one cannot have the uniqueness in Theorem \ref{THM.2}. Hence, in order to prove the unicity theorem, we need to show that $\BC[\cX]=\ZF$ in this case.  
A large part of our paper is to do exactly this, i.e. to prove that the center $Z_\zeta(F)$ is the {\em expected one}. For example, if $F$ is of finite type and the order of $\zeta$ is 2 mod 4, then we show that the center $\ZF$ is generated by $\BC[\cX]$ and the  skeins surrounding the punctures. A similar statement, though more complicated, is true for other roots of unity, see Theorem \ref{thm.center0}. After that, Theorem \ref{THM.2} follows fairly easily from Theorem \ref{THM.unicity}, see Section \ref{irreps}.  We also  give a formula for the dimension of the generic irreducible representations of $K_{\zeta}(F)$ (i.e. the number $N$ of Theorem \ref{THM.2}) when $F$ has at least one puncture, and the order of $\zeta$ is $\neq 0 \mod{4}$. In a future publication we will give a formula for $N$ for all cases and study the generic representations of $\KF$.

Since the Kauffman bracket skein algebra $K_{\zeta}(F)$  is a vector space with basis the simple diagrams on $F$, our proofs depend on parametrizations of the simple diagrams. In the case where $F$ has punctures, this is done via geometric intersection numbers with the edges of an ideal triangulation of $F$. If $F$ is closed we use Dehn-Thurston coordinates. These depend on a choice of a pants decomposition $\cP$ of the surface, and a dual graph $\cD$ embedded into the surface. In order to characterize the center of the Kauffman bracket skein algebra of a closed surface we prove Theorem \ref{twistedpants}, which says that given a finite collection of simple diagrams $\mathcal{S}$ on $F$ it is always possible to find $(\cP,\cD)$ so that the Dehn coordinates of all the diagrams in $\mathcal{S}$ are {\em triangular} in the sense that if $P_i,P_j,P_k$ are simple closed curves in $\cP$ that bound a pair of pants, then the geometric intersection numbers of any $S\in\mathcal{S}$ with $P_i,P_j$ and $P_k$ satisfy all triangle inequalities, and there are no curves in $\mathcal{S}$  that are isotopic to a curve in $\cP$.
                                  
\begin{remark}
 N. Takenov \cite{Takenov} had proved the unicity theorem for small surfaces, namely  spheres with  at most four punctures and  tori with at most one puncture, by using explicit presentations of the skein algebras of these surfaces. 
\end{remark}

The paper is organized as follows. We start  in Section \ref{unicitytheorem} by summarizing  results about  Azumaya algebras that lead to the proof of the unicity theorem. Section \ref{review}  reviews known facts and proves additional results about the Kauffman bracket skein algebra of a finite type surface and methods for parametrizing simple diagrams on a surface. 
In Section \ref{center} we characterize the center of the Kauffman bracket skein algebra of any finite type surface  at any  root of unity.
In Section \ref{finite} we prove that for any finite type surface, at any root of unity, the Kauffman bracket skein algebra is finitely generated as a module over its center, and describe the variety of classical shadows.  
In Section \ref{irreps} we prove Theorem \ref{THM.2} and give a formula for the number $N$ of Theorem \ref{THM.2} when $F$ has at least one puncture, and the order of $\zeta$ is $\neq 0 \mod{4}$. 

The authors thank Vic Camillo, Mio Iovanov, Ryan Kinser, Paul Muhly, Nikolaus Vonessen and Zinovy Reichstein for their help understanding Azumaya algebras and  for reading earlier versions of the algebraic arguments in this paper, and Julien March\'e,  Adam Sikora and Daniel Douglas for discussions on skein algebras. The authors also thank the referees for the remarks that helped us improve the paper.

This material is based upon work supported by and while serving at the National Science Foundation. Any
opinion, findings, and conclusions or recommendations expressed in this material are those of the authors
and do not necessarily reflect the views of the National Science Foundation.

\section {The Unicity Theorem}\label{unicitytheorem}
The goal of this section is to  prove the unicity theorem (Theorem \ref{THM.unicity}).  This is done by showing that a prime algebra which is finitely generated as a module over its center is Azumaya after appropriate localization.  We start with a review of some ring theory found in \cite{A,GS,MR}, and by summarizing relevant  definitions and results.
\subsection{Basic definitions}
Throughout this paper rings are assumed to be associative and have a unit, and  
ring homomorphisms  preserve the unit, unless otherwise stated. For a subset $X$ of a ring $R$, define
$ XR$ to be the right ideal generated by $X$. For a commutative ring $C$ let $M_n(C)$ denote the ring of $n\times n$ matrices with entries in $C$. By an {\bf ideal} we  mean a 2-sided ideal unless otherwise explicitly stated.

The ring $R$ is {\bf prime} if all  $a,b\in R$ satisfy the following condition: if  $arb=0$ for all $r\in R$ then $a=0$ or $b=0$. 
We say $R$ has {\bf no zero divisors} if for every $a,b\in R$, $ab=0$ implies $a=0$ or $b=0$. Clearly, if $R$ has no zero divisors then it is prime.
If $R$ is commutative and has no zero divisors it is an {\bf integral domain}. If $D$ is an integral domain, then for any positive integer $n$, the ring $M_n(D)$ of $n\times n$-matrices with coefficients in $D$ is prime. Prime rings are the noncommutative analogs of integral domains.  An ideal $I$ of $R$ is {\bf prime} if $R/I$ is a prime ring.

Denote the {\bf center} of the ring $R$  by $Z(R)$,
\begin{equation*}Z(R)=\{z\in R \mid \forall r \in R, \ rz=zr\}.\end{equation*}
If $R$ is a prime ring, then its center $Z(R)$ is 
 an integral domain.  Recall that a ring $R$ is an {\bf affine} $C$-algebra, where $C$ is a commutative ring, if $C\leq Z(R)$ and $R$  is finitely generated as an algebra over $C$.  We use the following  form of the Artin-Tate lemma that can be found in \cite[13.9.10]{MR}.
\begin{lemma} \label{artintate} Suppose a subring $C$ of the center $Z(R)$ is Noetherian. Assume that $R$ is both a finitely generated module over $Z(R)$ and an affine $C$-algebra.  Then $Z(R)$ is  also an affine $C$-algebra.

\end{lemma}

If $S$ is a multiplicatively closed subset  of $Z(R)$ that does not contain $0$ then we define the {\bf localization} of $R$ at $S$, denoted $S^{-1}R$, by
$$ S^{-1}R =  S^{-1} Z(R) \otimes _{Z(R)} R,$$
where $S^{-1} Z(R)$ is the standard localization of the commutative ring $Z(R)$ at $S$.   
If $S$ consists of powers of a single element $c\in Z(R)$ then the localization is denoted by $R_c$. The universal property of localization says that if $f: R \to R'$ is a ring homomorphism such that $f(c)$ is invertible, then $f$ has a unique ring homomorphism  extension from $R_c$ to $R'$.

\subsection{Irreducible representations and central characters}
Suppose $R$ is a $k$-algebra where $k$ is an algebraically closed field. A $k$-{\bf representation} of $R$ is a $k$-algebra homomorphism 
$$ \rho:R\rightarrow M_n(k) $$
 for some $n>0$. The number $n$ is called the dimension of $\rho$. 
 A representation is {\bf irreducible} if the induced action on $n$-dimensional column vectors has no nontrivial invariant subspaces, or equivalently if the homomorphism $\rho$ is onto. Two representations $\rho_1,\rho_2:R\rightarrow M_n(k)$ are {\bf equivalent} if there is an invertible matrix $x\in M_n(k)$ such that $\rho_1 = x \rho_2 x^{-1}$.
 The set of all equivalence classes of finite-dimensional irreducible $k$-representations of $R$ is denoted by $\Ir_k(R)$, or $\Ir(R)$ if there is no confusion.

By Wedderburn's theorem (see \cite[Theorem 2.61]{Bresar}), if a prime $k$-algebra $R$ is  finite-dimensional as a $k$-vector space, then $R \cong M_n(k)$ for some positive integer $n$, and this isomorphism is the only (up to equivalence) irreducible $k$-representation of $R$.

If $ \rho:R\rightarrow M_n(k) $ is an irreducible $k$-representation, then $I=\ker \rho$ is a prime ideal of $R$, since $R/I=M_n(k)$ is prime. Besides, $I$  is $k$-{\bf cofinite}, i.e.   the $k$-vector space $R/I$ has finite positive dimension. 
Conversely, if $I$ is a prime $k$-cofinite ideal of $R$, then Wedderburn's theorem shows that there is an irreducible $k$-representation $\rho$ with $\ker \rho =I$.

\def\Spec{\mathrm{Spec}}

The Skolem-Noether theorem implies that two irreducible $k$-representations $\rho_1, \rho_2$ are equivalent if and only if $\ker \rho_1= \ker \rho_2$. Thus we have the following bijection
\be 
\label{eq.bijection}
\Ir_k(R) \cong \{ I \in \Spec(R) \mid I \text{ is $k$-cofinite }\}, \quad \rho \to \ker \rho.
\ee
Here $\Spec(R)$ is the set of all prime ideals of $R$.

Suppose $\rho:R\rightarrow M_n(k)$ is an irreducible $k$-representation. By Schur's lemma, if $z\in Z(R)$ is a central element then $\rho(z)$ is a scalar multiple of the identity. This implies there exists a unique $k$-algebra homomorphism $\chi_\rho:Z(R)\rightarrow k$ such that 
\begin{equation} \forall z \in Z(R),\quad  \rho(z)=\chi_\rho(z)Id_n.\end{equation}
The algebra homomorphism $\chi_\rho$ is  called  the {\bf central character} of $\rho$.

Let $\HA(R,k)$ denote the set of all $k$-algebra homomorphisms from $R$ to $k$.
Since $\chi_\rho$ only depends on the equivalence class of $\rho$  we have the central character map
\begin{equation}
\label{chi}
\chi: \Ir_k(R) \to \HA(Z(R),k), \quad \rho \mapsto \chi_\rho.
\end{equation}
In general, $\chi$ is neither injective nor surjective. We will see that there is a nice class of algebras for which $\chi$ is a bijection. For $\tau\in \HA(Z(R),k)$, the preimage of $\tau$ under $\chi$ is denoted by $\Ir_k(R,\tau)$.

\begin{prop}\label{r.tau}
 Suppose $\tau \in \HA(Z(R),k)$ with $\fm =\ker \tau$.
Let $\pr: R \to R /\fm R$ be the natural projection. There is a bijection
\be  \label{eq.6h}
f_\tau: \Ir_k(R/\fm R) \overset {\cong}\longrightarrow \Ir_k(R,\tau), \quad \rho \to \rho \circ \pr.
\ee
\end{prop}
\def\tP{\tilde P}
\begin{proof} 
If $\rho: R/\fm R \to M_n(k)$ is an irreducible $k$-representation of $R/\fm R$, then $\rho \circ \pr$
 is an irreducible $k$-representation of $R$ since clearly $\rho\circ \pr$ is surjective. If $\rho_1$ and $\rho_2$ are two equivalent irreducible 
 $k$-representation of $R/\fm R$, then $\rho_1 \circ \pr$ and $\rho_2 \circ \pr$ are equivalent. 
 Hence the map $f_\tau$ given by \eqref{eq.6h} is well-defined. 
 
 Note that if $P$ is a prime ideal of $R/\fm R$ then $\tP:=\pr^{-1}(P)$ is a prime ideal of $R$, since $R/\tP \cong (R/\fm R)/P$. Besides $\pr(\tP)= P$. If $\rho$ is an irreducible $k$-representation of $R/\fm R$, with $P=\ker \rho$, then $\ker (\rho \circ \pr)= \tP$.

Suppose $\rho_1, \rho_2$  are two irreducible $k$-representations of $R/\fm R$ with $f_\tau(\rho_1)=f_\tau(\rho_2)$. Let $P_i= \ker \rho_i$, then $\tP_1= \ker (\rho_1 \circ \pr) = \ker (\rho_2 \circ \pr)= \tP_2$. It follows that $P_1=P_2$, as $P=\pr(\tP)$. This shows $\rho_1$ and $\rho_2$ are equivalent, or $f_\tau$ is injective.

Now suppose $\rho$ is an irreducible $k$-representation of $R$ with $\chi_\rho=\tau$. Then $\rho=0$ on $\ker \tau= \fm$. Hence $\rho=0$ on $ \fm R$. This means $\rho$ factors through $R/\fm R$, and hence $f_\tau$ is surjective. The proposition is proved.
\end{proof}

Suppose that $Z(R)$ is an affine $k$-algebra without nilpotents, i.e. if $x\in Z(R)$ and $x^n=0$ for some positive integer $n$ then $x=0$. Let $\MS(Z(R))$ be the set of all maximal ideals of $Z(R)$. 
The weak Nullstellensatz \cite{AM}  says that  $\MS(Z(R))$  is an affine algebraic set whose ring of regular function is $Z(R)$, and the map
\begin{equation}\label{tau1}
\eta: \HA(Z(R),k) \to \MS(Z(R)),\  \eta( \phi) = \ker \phi,
\end{equation}
is a bijection. 
We can pull back the Zariski topology of $\MS(Z(R))$ to $\HA(Z(R),k)$.

\subsection{PI rings, central simple algebras, and Razmyslov polynomials}
The ring of noncommuting polynomials in variables  $x_1,\ldots,x_n$  with integer coefficients is denoted $\mathbb{Z}\langle x_1,\ldots,x_n\rangle $.  
A polynomial $p$ is {\bf linear} in the variable $x_i$ if every monomial that appears in $p$ with nonzero coefficient has total exponent $1$ in $x_i$.  A polynomial is {\bf multilinear} if it is linear in all the variables. 
 A polynomial $p$ is {\bf monic} if at least one of its terms with highest total exponent sum has coefficient $\pm 1$. For $f\in \mathbb{Z} \langle x_1,\ldots,x_n \rangle$ and an associative ring $R$ let $f(R)\subset R$ be the set
 $$ f(R)= \{ f(r_1,\ldots,r_n) \mid r_1, \dots, r_n \in R\}.$$ 
  We say that $f\in \mathbb{Z} \langle x_1,\ldots,x_n \rangle$ is a {\bf polynomial identity} for the ring $R$ if $f(R)=\{0\}$.   If  $R$ has a monic polynomial identity then $R$ is a {\bf PI ring}. For example,  if $R$ is finitely generated as a module over its center, then $R$ is a PI ring, see \cite[Corollary 13.1.13]{MR}.
  A  homogeneous polynomial $f\in \mathbb{Z} \langle x_1,\ldots,x_n \rangle$ of positive degree is a {\bf central polynomial} for the ring $R$ if $f(R)\subset Z(R)$ and $f$ is not a polynomial identity, i.e. $f(R) \neq \{0\}$.

 An ring $R$ is called a {\bf central simple algebra} 
 if it has no nontrivial two-sided ideals, its center $Z(R)$ is a field, and it is finite-dimensional over $Z(R)$. Usually $R$ is considered as an algebra over $Z(R)$, and whence the terminology.  If $R$ is a central simple algebra and  $\bZR$ denotes the algebraic closure of $Z(R)$, then there exists a positive integer $n$, called the {\bf PI degree of $R$},  such that $R \otimes _{Z(R)} \bZR\cong M_n(\bZR)$. In particular, if $Z(R)=k$ is  an algebraically closed field,  then $R$ is isomorphic to $M_n(k)$ with $n$ being the PI degree,  and this isomorphism is the only (up to equivalence) irreducible $k$-representation of $R$. It also follows that for any central simple algebra, its dimension over the center is a perfect square, namely, the square of its PI degree.
 For details, see \cite[Chapter 13]{MR}.

If $R$ is a central simple algebra of PI degree $n$, by \cite[Cor 13.6.3]{MR} the {\bf $n$th Razmyslov polynomial} $g_n$, defined in \cite{A,MR}, is a   central polynomial for $R$ of degree $2n^2+2$ that is multilinear. 

We will use the following well-known results concerning prime PI ring.
 
 \begin{theorem} \label{posner} Suppose $R$ is a prime PI ring. Let $S=Z(R)-\{0\}$.  
 
 (a) The ring
$S^{-1}R$ is a central simple algebra with center $S^{-1}Z(R)$. The PI degree of $R$ is defined to be the PI degree of $S^{-1}R$.

(b) For every prime ideal $I$ of $R$, the ring $R/I$ is a prime PI ring whose PI degree is less than or equal to that of $R$.
\end{theorem}
Part (a) is Posner's theorem,   see \cite{A} or \cite[Theorem 13.6.5]{MR}. Part (b) is \cite[Lemma 17.7.2(i)]{MR}.

\subsection{Azumaya algebras}  The concept of an Azumaya algebra  generalizes the concept of a central simple algebra to the case when the center is not necessarily  a field. 

If $R$ is a ring let  $R^o$ be the opposite ring, i.e. $R^o$ has the same underlying additive group, but multiplication is defined to be the opposite of the multiplication on $R$.  There is a ring homomorphism \begin{equation} \Psi : R\otimes_{Z(R)} R^o\rightarrow End_{Z(R)}(R)\end{equation} given
by \begin{equation}\Psi(\alpha\otimes \beta) (\gamma)=\alpha\gamma\beta.\end{equation}
A ring $R$ is an {\bf Azumaya algebra}  if $R$ is a finitely generated  projective module over
$Z(R)$ and $\Psi$ is an isomorphism. Usually one considers $R$ is an algebra over $Z(R)$, but it may happen that $Z(R)$ contains a field $k$, and we can consider both $R$ and $Z(R)$ as algebras over $k$.

An important property of Azumaya algebras, 
see \cite[Proposition 13.7.9]{MR},  is that there is a bijection between 2-sided ideals $I$ of $R$ and ideals $H$ of $Z(R)$ given by 
\begin{equation}I \to I\cap Z(R), \ \ H \to HR.
\label{eq.bij}\end{equation}

We are interested in Azumaya algebras with the property that all  irreducible $k$-representations  have the same dimension. 
Suppose $R$ is an Azumaya algebra. By \cite[Proposition 13.7.9]{MR}, if  $\fm\subset Z(R)$ is a maximal ideal, then $R/\fm R$ is a central simple algebra with center $Z(R)/\fm$. Following~\cite{MR}, we say $R$ is an {\bf Azumaya algebra of rank $n^2$} if for each maximal ideal $\fm \subset Z(R)$, the central simple algebra $R/\fm R$ has PI degree $n$.

\def\Id{\mathrm{Id}}
\begin{prop} \label{iso}
Suppose that $R$ is an Azumaya algebra of rank $n^2$. Assume that the center $Z(R)$  is an affine  $k$-algebra, where $k$ is an algebraically closed field. It follows that:

(i) The central character map $\chi: \Ir_k(R) \to \HA(Z(R),k)$ is a bijection;

(ii)  Every irreducible $k$-representation of $R$ has dimension $n$.
\end{prop}
\begin{proof}

  Suppose $\tau\in \HA(Z(R),k)$.  Then $\fm=\ker \tau$ is a maximal ideal of $Z(R)$. By Proposition \ref{r.tau}, $\Ir_k(R,\tau)= \Ir_k(R/\fm R)$. Since $R/\fm R$ is a central simple algebra having PI degree $n$ and center $Z(R)/\fm=k $, an  algebraically closed field, we have  $R/\fm R \cong M_n(k)$, and this is the only irreducible representation of $R/\fm R$. This shows $\Ir_k(R,\tau)$ has exactly one element, proving (i). Besides, the dimension of the representation in $\Ir_k(R,\tau)$ is $n$, proving (ii).
  \end{proof}

By \cite[Theorem 13.7.14]{MR}, we have the following criterion for a prime ring to be Azumaya of rank $n^2$. Recall that $g_n$ is the  $n$th Razmyslov polynomial.
\begin{theorem}[Artin-Procesi theorem] \label{ap} If $R$ is a prime ring then it is an Azumaya algebra of rank $n^2$ if and only if $g_n(R) R=R$.
\end{theorem}

We say that an algebra $R$ is   {\bf almost Azumaya} if there is $c\in Z(R)\setminus \{0\}$ such that  $R_c$ is an Azumaya algebra of rank $n^2$ for some integer $n>0$. 

We are unsure who to attribute the next theorem to, it appears as Corollary V.9.3 on page 73 of \cite{A} with a different proof. It also appears in \cite{R} as Corollary 6.1.36. 

\begin{theorem}\label{aa}  If $R$ is a prime algebra that is finitely generated as a module over its center then $R$ is almost Azumaya.\end{theorem}

\proof  By Posner's theorem (Theorem \ref{posner}(a)), $S^{-1}R$ is a central simple algebra, where $S= Z(R) \setminus \{0\}$. Suppose $n$ is the PI degree of $S^{-1}R$. 
As the Razmyslov polynomial $g_n$ is central, there is a tuple ${\mathbf x}$ of elements in $S^{-1}R$ such that $g_n({\mathbf x}) \neq 0$. Since $g_n$ is multilinear \cite{A,MR} over the center, $g_n({\mathbf x})$ can be written as a central linear combination of evaluations of $g_n$ on elements of $R$. Since the sum is nonzero, one of the terms gives a nonzero evaluation of $g_n$ on elements of $R$.  Let $c=g_n({\mathbf y})$ be one of the nonzero evaluations, where ${\mathbf y}$ is a tuple of elements from $R$. The ring $R_c$ is still prime. Since $c\in g_n(R_c)$ is invertible, we have $g_n(R_c) R_c= R_c$. By Theorem \ref{ap}, $R_c$ is an Azumaya algebra of rank $n^2$. 
 \qed

\subsection{Proof of the Unicity Theorem}
Suppose   $R$ is a prime  affine $k$-algebra, where $k$ is an algebraically closed field and $R$ is finitely generated as a module over $Z(R)$. By the Artin-Tate lemma (Lemma \ref{artintate}, with $C=k$),  $Z(R)$ is an affine $k$-algebra. Since $R$ is prime, $Z(R)$ is a domain. It follows that $\MS(Z(R))$ is an irreducible affine algebraic set over $k$ whose ring of regular functions is $Z(R)$. The Zariski topology of $\MS(Z(R))$  is generated by open sets of the form
$$ U_c:=\{\mathfrak{m}\in \operatorname{Max}\:\operatorname{Spec} \left({Z(R)}\right) \mid c\not \in \mathfrak{m}\}, $$
where $c \in Z(R)$. By pulling back using $\eta$ of \eqref{tau1}, the Zariski topology of $\HA(Z(R),k)$ is generated by open sets of the form
$$V_c:=\{f \in \HA(Z(R),k) ) \mid f(c) \neq 0\}, \quad c \in Z(R).
$$

We now reformulate Theorem \ref{THM.unicity} in a more precise form.

\def\embed{\hookrightarrow}
\begin{theorem}
\label{unicity} Let $k$ be an algebraically closed field and 
 $R$ be a prime affine  
$k$-algebra. Suppose $R$ is generated  as a module over its center $Z(R)$ by a finite set  of $r$ elements.

(i) Every element of $\HA(Z(R),k)$ is the central character of at least one irreducible $k$-representation and at most $r$ non-equivalent irreducible $k$-representations.

(ii) Every irreducible $k$-representation of $R$ has dimension $\le N$, which is the PI degree of $R$, and also the square root of the rank of $R$ over $Z(R)$.

(iii) There exists a  Zariski open and dense subset of the form  $V_c$ of \  $\HA(Z(R),k)$, where $0\neq c \in Z(R)$, such that every $\tau\in V_c$ is the central character of a unique (up to equivalence) irreducible $k$-representation $\rho_\tau$. Moreover all  representations $\rho_\tau$ with $\tau\in V_c$ have  dimension $N$.
\end{theorem}

\def\brho{\bar \rho}
\def\btau{\bar \tau}
\proof

Since $R$ is finitely generated  as a module over its center $Z(R)$, it is a PI ring. Moreover, $R$ is prime; hence its PI degree $N$ is defined, and is equal the square root of the rank of $R$ over $Z(R)$.

(i) Suppose $\tau\in \HA(Z(R),k)$ and 
$\fm=\ker \tau$. Note that $\fm \neq Z(R)$ since $\tau$ preserves the unit and cannot be the zero map.

Since a generating set of $R$ over $Z(R)$ projects down to a generating set of $R/\fm R$ over $k= Z(R)/\fm$,  the $k$-dimension of $R/\fm R$ is $\le r$. As a  $k$-algebra whose $k$-dimension is $\le r$, the algebra $R/\fm R$ has at most $r$ non-equivalent irreducible $k$-representations. By Proposition \ref{r.tau} one has $|\Ir_k(R,\tau)| =|\Ir_k(R/\fm R)| \le r$.

Now we prove that $|\Ir_k(R,\tau)| \ge 1$. 
Finite generation of $R$ as a module over $Z(R)$  implies that $R$ is integral over $Z(R)$, which means every element in $R$ is a root of a monic polynomial in one variable with coefficients in $Z(R)$, see \cite[Corollary 13.8.9]{MR}. By the ``lying over" property (see  \cite[Theorem 13.8.14]{MR}), there exists a prime ideal $I$ of $R$ such that 
$I \cap Z(R) =\fm$. This 
identity  shows that the embedding $Z(R) \embed R$ descends to an embedding of 
$Z(R)/\fm$ (which is $k$) into $R/I$. Thus $R/I$ is a non-zero $k$-algebra. Besides $R/I$ is prime, and finite dimenisonal over $k$.
By Wedderburn's theorem (see \cite[Theorem 2.61]{Bresar}), $R$ is isomorphic to $M_s(k)$ for some positive integer $s$. The isomorphism between $R/I$ and $M_s(k)$ gives an irreducible $k$-representation whose central character is $\tau$. 

(ii) Suppose $\rho: R \to M_s(k)$ is an irreducible $k$-representation. Let  $I=\ker \rho$, then $R/I\cong M_s(k)$. Hence $s$ is the PI degree of $R/I$.  By \cite[Lemma 13.7.2]{MR}, the PI degree of $R/I$, for any prime ideal $I$, is less than or equal to the PI degree of $R$. It follows that $s \le N$.

(iii) 
By Theorem \ref{aa} there exists a nonzero $c\in Z(R)$ such that $R_c$ is Azumaya of rank $N^2$. We have the  inclusion $ R \hookrightarrow R_c$, and $R_c$ is generated as an algebra by $R$ and $c^{-1}$.
Note that $Z(R_c)= Z(R)_c$, the localization of $Z(R)$ at $c$. 

For a $k$-algebra homomorphism $\tau: Z(R_c) \to k$ let $\btau$ be its restriction to $Z(R)$. If $\btau_1=\btau_2$, then $\tau_1(c^{-1})= \tau_2(c^{-1})$ since both are the inverse of $\btau_1(c)= \btau_2(c)$. Hence the map
$\iota: \HA(Z(R_c)) \to Z(R)$ defined by $\iota(\tau)= \btau$ is injective. The image if $\iota$ is exactly $V_c$. 

Suppose $\mu\in V_c$. Then  $\mu=\btau$ for some $\tau\in \HA(Z(R_c))$. Since $R_c$ is Azumaya of rank $N^2$,
 there is an irreducible $k$-representation $\rho: R_c \to M_N(k)$, see Proposition \ref{iso}. 
 Let $\brho$ be the restriction of $\rho$ to $R$. As $c^{-1}$ is a central element, $\rho(c^{-1})$ is a scalar multiple of the identity, hence any subspace of the representation space is invariant under $\rho(c^{-1})$. It follows that any invariant subspace of $\brho$
  is also an invariant subspace of $\rho$, which does not have non-trivial invariant subspaces. Thus, $\brho$ is also irreducible.
   Clearly $\chi_{\brho}=\btau=\mu$, 
   and the dimension of $\brho$ is $N$.
    Thus, any element $\mu\in V_c$ is the central character of at least one irreducible $k$-representation. Besides, the dimension of this $k$-representation is $N$.

Suppose $\nu: R \to M_s(k)$ is another irreducible $k$-representation of $R$ with $\chi_\nu=\mu$.  
Then $\nu(c)= \mu(c) \Id_s$ is a non-zero scalar multiple of the identity, and hence is invertible. The universal property of localization implies that $\nu$ can be extended to a $k$-algebra homomorphism $\tilde \nu: R_c \to M_s(k)$, which is surjective since $\nu$ is. Thus, $\tilde \nu$ is an irreducible $k$-representation of $R_c$ whose central character is $\tau$. By Proposition \ref{iso} 
there is a unique (up to equivalence) irreducible $k$-representation whose central character is $\tau$. Hence $\tilde \nu$ is equivalent to $\rho$. It follows that $\nu$ is equivalent to $\brho$. Thus, every $\mu\in V_c$ is the central character of exactly one (up to equivalence) irreducible $k$-representation of $R$, and the dimension of the representation is $N$.
\qed

\def\bF{\bar F}
\section{The structure of the Kauffman Bracket skein algebra}\label{review}

Throughout the paper $\BN$, $\BZ$ and $\BC$ are respectively the set of non-negative integers, the set of integers, and the set of complex numbers. Let  $\bf i$ denote the complex unit. A complex number $\zeta$ is called a root of 1 if $\zeta^n=1$ for some positivie integer $n$, and the smallest such $n$ is called the order of $\zeta$ and denoted by $\ord(\zeta)$.

In this section we recall some known facts and prove additional results about the  Kauffman bracket  skein algebra of a finite type surface. We also develop techniques for parametrizing simple diagrams. New results include Theorem \ref{thm.Keven} about canonical isomorphisms of the even skein algebra at a root of 1 of order 4, Proposition \ref{r.com} about the commutation of two simple diagrams, and Theorem \ref{twistedpants} about the triangular covering of Dehn-Thurston coordinates.  

\subsection{The Kauffman bracket skein algebra} Throughout we fix a {\bf finite type surface} $F$, i.e. a surface of the form $F= \bF \setminus \cV$, where $\bF$ is a connected, closed,  oriented surface and $\cV$ is a (possibly empty) finite set. The surface $\bF$ can be uniquely recovered from $F$.
A loop in $F$  bounding a  disk in $\bF$ containing exactly one point in $\cV$ is called a {\bf peripheral loop}.

A {\bf framed link}  in  $F\times [0,1]$ is an embedding of  a disjoint union of oriented annuli in  $F\times [0,1]$. Framed links are usually considered up to isotopy. By convention the empty set is considered as a framed link with 0 components and is isotopic only to itself.
The orientation of a component of a framed link corresponds to choosing a preferred side to the annulus. If $D$ is a non-oriented link diagram on $F$ then a regular neighborhood of $D$ is a framed link with the preferred side up. We identify $D$ with the isotopy class of the framed link it defines.  Any framed link in $F \times [0,1]$ is isotopic to a framed link determined by a link diagram. A {\bf simple diagram} on  $F$ is a link diagram with no crossings and no {\bf trivial loops}, i.e. a curve bounding a disk.  We consider the empty set as a simple diagram which is isotopic only to itself.

For a  non-zero complex number $\zeta$, the Kauffman bracket skein module of $F$ at $\zeta$, denoted by $\KF$, is the $\BC$-vector space freely spanned by all isotopy classes of framed links in $F \times [0,1]$ subject to the following {\bf skein relations}
\begin{align}
\lcr &= \zeta \zer+ \zeta ^{-1}\ift   \label{KBSR}\\
\bigcirc \sqcup L  &=  (- \zeta^2 -\zeta^{-2}) L\label{KBSR2}.
\end{align}
Here the framed links in
each expression are identical outside the balls pictured in the diagrams.
The first (resp. second) diagram on the right hand side of \eqref{KBSR} will be referred to as the $+1$ (resp. $-1$) {\bf smooth resolution} of the left hand side diagram at the crossing.

 The algebra structure of $K_\zeta(F)$ comes from stacking.  More precisely, the product of two links is defined by placing the first link above the second in the direction given by the interval $[0,1]$. The product descends distributively to a product on the skein module.  Note that there are non-homeomorphic surfaces $F$ and $F'$ with $F \times [0,1] \cong F'\times [0,1]$. The isomorphism gives rise to an isomorphism between
  $\KF$ and $K_\zeta(F')$  as vector spaces, but as algebras they are generally  not isomorphic.

\def\sS{\mathscr S}

 Suppose $D$ is a link diagram on $F$. Let $\cC$ be the set of all crossings. For every map $\sigma: \cC \to \{\pm1 \}$ let $D_\sigma$ be the non-crossing diagram obtained from $D$ by doing the  $\sigma(c)$ smooth resolution at every crossing $c$.
Let $|\sigma|= \sum_{c\in \cC}  \sigma(c)$ and $l(\sigma)$ be the number of trivial components of $D_\sigma$. Let $D_\sigma'$ be the simple diagram obtained from $D_\sigma$ by removing all the trivial components. Using the two  skein relations, we have 
\be 
\label{eq.prod}
D = \sum_{\sigma: \cC \to \{\pm 1\}} \zeta^{|\sigma|} (-\zeta^2 - \zeta^{-2})^{l(\sigma)} \, D_\sigma' \quad \text{in } \KF.
\ee
 
 This shows the set $\sS$ of all isotopy classes of simple diagrams spans $\KF$. A little more work will show  the following.

\begin{prop}[see \cite{PS}] 
The set $\sS$ of all isotopy classes of simple diagrams
 forms a basis for $K_{\zeta}(F)$ as a vector space over $\mathbb{C}$.
\end{prop}

Note that the same set $\sS$ serves as a vector space basis for $\KF$ for all non-zero $\zeta$. It is the algebra structure which depends on $\zeta$.

Every  $0\neq \al \in \KF$ has a unique {\bf standard presentation}, which is a finite sum
\be
\label{eq.pres} 
\al = \sum_{j\in J} c_j \al_j,
\ee
where $c_j\neq 0$ and $\al_j\in \sS$ are distinct. 

\begin{remark}Kauffman bracket skein modules for oriented 3-manifolds were introduced independently by Turaev and Przytycki \cite{Prz,Turaev} in an attempt to generalize the Jones polynomial to general 3-manifolds. The algebra structure of $\KF$  was first introduced by Turaev.
\end{remark}

\def\ev{{\mathrm{ev}}}
\def\Kevz{K^\ev_\zeta(F)}
\subsection{Even skeins and the case $\zeta\in \{\pm 1, \pm \bf i\}$}\label{sec.even}

 When $\zeta=\pm 1$ the algebra $K_{\zeta}(F)$ is commutative. The algebra $K_{-1}(F)$  can be canonically identified with the coordinate ring $\BC[\cX(F)]$ of the $SL_2\mathbb{C}$-character variety $\cX(F)$ of the fundamental group of $F$, \cite{B,PS}. Barrett \cite{Ba} constructed an algebra isomorphism from $\KF$ to $K_{-\zeta}(F)$  for every spin structure of $F$. Hence
$K_1(F)$ is also isomorphic to $\BC[\cX(F)]$, although the isomorphism is not canonical. However,  $K_1(F)$ can be canonically identified with the coordinate ring of a twisted $SL_2(\BC)$-character variety of the fundamental group of $F$, \cite{BW0,Thurston}.

Recall that if $\alpha$ and $\beta$ are properly embedded $1$--manifolds in $F$, and at least one of $\alpha$ or $\beta$ is compact, their {\bf geometric intersection number} $i(\alpha,\beta)$  is the minimum number of points in $\alpha'\cap \beta'$ over all properly embedded $1$-manifolds $\alpha'$ and $\beta'$ that are isotopic to $\alpha$ and $\beta$ via a compactly supported isotopy, and are transverse to one another. 
Let $\BZ_2=\BZ/2\BZ$. Define
$i_2(\al,\beta)\in \BZ_2$ by $i_2(\al,\beta)= i(\al,\beta) \pmod 2$.

A simple diagram $\al$, or its isotopy class, is called {\bf even} if $i_2(\al,\beta)=0$  for all simple loops $\beta$.
Let $\sS^\ev\subset \sS$ be the subset consisting of even elements, and  $\Kevz$ be the $\BC$-subspace of $\KF$ spanned by $\sS^\ev$.

Although the algebras $K_\zeta(F)$ and  $K_{-\zeta}(F)$ are not canonically isomorphic,  we have the following.
\begin{prop}  \label{r.sign}
 (a) The susbspace $\Kevz$ is a subalgebra of $\KF$.

(b) There is a canonical algebra isomorphism $h: \Kevz \to K^\ev_{-\zeta}(F)$ given by $h(\al)=\al$ for all $\al\in \sS^\ev$.
\end{prop}
\begin{proof} Suppose $\al, \al'$ are even simple diagrams. By isotopy we assume $\al$ is transversal to $\al'$. The product $\al \al'$ is presented by the link diagram $D=\al \cup \al'$ with $\al$ above $\al'$, and can be calculated by \eqref{eq.prod}. 

(a) Since  a smooth resolution does not change the homology class in $H_1(F;\BZ_2)$, for every $\sigma$ of \eqref{eq.prod} we have $D_\sigma' =\al + \al '$ in $H_1(F;\BZ_2)$. Hence for every simple diagram~$\beta$,
$$ i_2(D_\sigma',\beta) =i_2(\al, \beta) + i_2(\al', \beta) =0.$$

This shows each $D_\sigma' \in \sS^\ev$, and $\al\beta \in \Kevz$.

b) Since $\al,\al'$ are even, there is an even number of crossings. It follows that $|\sigma|$ is always even. Hence the right hand side of \eqref{eq.prod} remains the same if we replace $\zeta$ by $-\zeta$. This shows $h$ is an algebra homomorphism. Since $h^2$ is the identity, it is invertible. This proves (b).
\end{proof}

\def\bA{\bar {\cA}}
\def\bF{\bar F}
\def\bH{\bar H}
\def\bv{\bar \varphi}

By \cite{M,Si}, when $\zeta =\pm\bf i$, the skein algebra $\KF$ is a twisted version of $K_{-1}$. We will modify this result in the following form.

Recall that if $\al$ and $\beta$ are oriented simple diagrams on $F$, then  the {\bf algebraic intersection index} $\omega(\al ,\beta) \in \BZ$ is the number of intersections of $\al$ and $\beta$, counted with sign (after an isotopy to make  $\al$ and $\beta$ transversal). Here the sign of an intersection point is positive if the tangents of $\al$ and $\beta$ at that point form the positive orientation of the surface, and it is negative otherwise. Unlike the geometric intersection index, both the algebraic intersection index and $i_2$ have homological flavor: 
 the algebraic intersection index is an anti-symmetric bilinear form 
 \be \omega: H_1(F;\BZ) \otimes_\BZ H_1(F;\BZ) \to \BZ \label{eq.form}\ee
  and $i_2$ is its reduction modulo 2. Besides, $\omega(\al,\beta)$ in $F$ is equal to $\omega(\al,\beta)$ in $\bF$. The kernel of the form $\omega$ is the subgroup $H_1^\partial(F;\BZ)$ generated by peripheral loops. The quotient $\bH_1(F;\BZ):= H_1(F;\BZ)/H_1^\partial(F;\BZ)$ is canonically isomorphic to $H_1(\bF;\BZ)$. The form $\omega$ descends to a bilinear form on $\bH_1(H,\BZ)$.
  
 Let $\bA$ be the $\BC$-algebra generated by symbols $[\gamma]$ for each $\gamma\in \bH_1(F;\BZ)=H_1(\bF;\BZ)$, subject to the following relations
\be 
\label{eq.defA}
[\gamma]^2=1, \quad [\gamma][\gamma'] = {\bf i} ^{\omega(\gamma', \gamma)} [\gamma+ \gamma'].
\ee
This algebra was introduced in \cite{BHMV}  and also studied in \cite{M}. 
 As a $\BC$-vector space, $\bA$ has dimension $2^{b_1(\bF)}$. In particular, it is not a zero vector space. For a simple diagram $\al$, define $[\al]=[\tilde \al]$, where $\tilde \al$ is $\al$ equipped with an arbitrary orientation. Since the components of $\al$ are disjoint, the defining relations \eqref{eq.defA} show that $[\al]$ does not depend on the choice of the orientation.

The following is a modification of an important result of \cite{M}.
\begin{theorem} \label{thm.Marche}
The $\BC$-linear map  $\bv: K_{-\bf i}(F)\to K_{-1}(F) \otimes_\BC \bA$ defined by
$\bv(\al) = (-1)^{l(\al)} \al \otimes [\al]$ for all $\al\in \sS$, where $l(\al)$ is the number of components of $\al$, is an injective algebra homomorphism.
\end{theorem}
\begin{proof} Let $\cA$ and $\varphi$  be  defined exactly like $\bA$ and $\bv$,  only with $\bH(F;\BZ)$ replaced by $H(F;\BZ)$. The theorem, with $\bA$ and $\bv$ replaced by $\cA$ and $\varphi$, was 
\cite[Theorem 3.2]{M}. 

Since the algebraic intersection index on $\bH_1(F;\BZ)$ comes from that on $H_1(F;\BZ)$, the map $\psi: \cA \to \bA$ given on generators by $\psi(\al) =\bar \al$ for $\al\in H_1(F;\BZ)$, where $\bar \al$ is the image of $\al$ under $H_1(F;\BZ) \to \bH_1(F;\BZ)$, is a well-defined algebra homomorphism. It follows that $\bv= (\mathrm{id} \otimes \psi) \circ \varphi$, a composition of two algebra homomorphisms,  is an algebra homomorphism.

Suppose $\al\in K_{-\bf i}(F)$ such that $\bv(\al)=0$. Let $\al=\sum_{j\in J} c_j \al_j$ be the standard presentation \eqref{eq.pres}. Then 
\be 
\label{eq.0g}
\sum_{j\in J} (-1)^{l(\al_j)}\, c_j \al_j \otimes [\al_j]=0 \quad \text{in } \ K_{-1}(F)\otimes \bA.
\ee
As $\sS$ is a $\BC$-basis of $K_{-1}(F)$, we have a  direct sum decomposition $K_{-1}(F)\otimes \bA= \bigoplus_{\beta\in \sS} \beta \otimes \bA$. Hence \eqref{eq.0g} implies that for each $j$ one has $[\al_j]=0$, which is impossible since $[\al_j]^2=1$. Thus $\bv(\al) \neq 0$, and the map $\bv$ is injective.
\end{proof}

\begin{remark} Actually in \cite{M}, J. March\'e proved Theorem 3.2 only for closed surfaces, but the proof does not use the fact that the surface is closed and works for all finite type surfaces. We would like to thank  March\'e for confirming this.
\end{remark}

  \def\tal{\tilde \al}
  
  It turns out that $\Kevz$ for all $\zeta\in \{\pm1, \pm \bf i\}$ are the same.
 \begin{theorem} \label{thm.Keven}
 Suppose $\zeta\in \{\pm1, \pm \bf i\}$ and $F$ is a finite type surface. Then there is a canonical algebra isomorphism $f: K^\ev_\zeta(F)\to K^\ev_{-1}(F)$, defined for $\al\in \sS^\ev$ by
$$ f(\al) = \begin{cases}  \al \quad &\text{if } \ \zeta =\pm 1 \\
(-1)^{l(\al)} \,\al \quad &\text{if } \ \zeta =\pm {\bf i},
\end{cases}  
$$
where $l(\al)$ is the number of components of $\al$. 
 \end{theorem} 
 \begin{proof} Proposition \ref{r.sign} proves the theorem for the case $\zeta=1$ and  also shows that the case $\zeta=\bf i$ follows from the case $\zeta=-\bf i$, which we assume from now on.
 
\begin{lemma} If $\al\in \sS^\ev$, then $[\al]=1$ in $\bA$.
\end{lemma} 
\begin{proof} Let $\tal$ be $\al$ with an arbitrary orientation. Note that $\al$ being even means $\al=0$ in the homology group $\bH_1(F;\BZ_2)= H_1(\bF;\BZ_2)$. We have the following exact sequence
$$ 0\to 2 \bH_1(F;\BZ) \to \bH_1(F;\BZ) \overset j {\onto} \bH_1(F;\BZ_2)\to 0.$$
Since $j(\tal)=\al =0$ in $\bH_1(F;\BZ_2)$, from the exactness we have $\tal= 2 \al'$ in $\bH_1(F;\BZ)$ for some $\al'\in \bH_1(F;\BZ)$.  Then $[\al]=[\tal] =[\al']^2 =1$.
\end{proof}
The lemma and the definition of $\bv$ (of Theorem \ref{thm.Marche}) gives
$$ \bv(\al) = (-1)^{l(\al)} \al \otimes [\al] = (-1)^{l(\al)} \al \otimes 1 \quad \forall \al \in \sS^\ev.$$
Since $\sS^\ev$ is a $\BC$-basis of both $\Kevz$ and $K^\ev_{-1}(F)$, 
this shows the algebra homomorphism $\bv$ maps $\Kevz$ isomorphically onto $K^\ev_{-1}\otimes 1$, which is isomorphic (as algebras) to $K^\ev_{-1}$. Combining the two algebra isomorphisms, we see that $f$ is an algebra isomorphism.
 \end{proof}  
 
We give a geometric description of $\Kevz$ for $\zeta \in \{\pm 1, \pm \bf i\}$ in Section \ref{finite}.

\subsection{The threading map, finite generation, and zero divisors}

 The {\bf Chebyshev polynomials of the first kind} are defined recursively by letting $T_0(x)=2$, $T_1(x)=x$ and  $T_k(x)=xT_{k-1}(x)-T_{k-2}(x)$. 
  They  
  satisfy the  {\bf product to sum formula}:
\begin{equation}\label{prodsum} T_k(x)T_l(x)=T_{k+l}(x)+T_{|k-l|}(x).\end{equation}

Parts (a) and (b) of the following theorem are important results of   Bonahon and Wong \cite{BW1}. Part (c) is a slight extension proved in \cite{Le}.

\def\ep{\epsilon}
\begin{theorem}[\cite{BW1,Le}]\label{imageofCH}
Let $F$ be a finite type surface and $\zeta$ be root of unity of order $n$. Let  $m=\ord(\zeta^4)=\frac{n}{gcd(n,4)}$ and $\epsilon=\zeta^{m^2}$.

(a) There exists a unique algebra homomorphism
\begin{equation}
Ch:K_{\epsilon}(F)\rightarrow K_{\zeta}(F)
\end{equation}
such that if $\al$ is a simple loop on $F$ then $Ch(\al)= T_m(\al)$. Moreover $Ch$ is injective.

(b) If  $n\not  \equiv 0 \pmod{4}$,   then $Ch(K_{\epsilon}(F))$ is contained in the center of $K_{\zeta}(F)$.

(c) If $n  \equiv 0 \pmod{4}$,  then for any simple diagrams $\al$ and $\beta$, one has
\begin{equation}\label{Cha}
Ch(\alpha)\,\beta=(-1)^{i(\al,\beta)}\beta\, Ch(\alpha). \end{equation}
Consequently, $Ch(K^\ev_\ep(F))$ is contained in the center of $K_{\zeta}(F)$.
\end{theorem}

Hence we can think of $K_{\epsilon}(F)$ as a subalgebra of the center of 
$K_{\zeta}(F)$  when $n\neq 0 \pmod{4}$, and  $K_{\epsilon}^\ev(F)$ as  a subalgebra of the center of 
$K_{\zeta}(F)$ when $n=0\pmod{4}$.

It was proved in \cite{Bullock} that the skein algebra $\KF$  is {\bf affine} over $\BC$, i.e. it is  finitely generated as an algebra over $\BC$. This result is strengthened in 
\cite{AF2} as follows. 
\begin{theorem}[\cite{AF2}] \label{spanning}If $F$ is a surface of finite type,  then there exist simple closed curves $J_1,\ldots,J_s$ on $F$  so that the collection of skeins 
 $T_{k_1}(J_1) T_{k_2}(J_2) \dots T_{k_s}(J_s)$,   where the $k_i$ range over all natural  numbers, spans $\KF$. \end{theorem}

Another important property of $\KF$ is the following.
\begin{theorem} [\cite{BW0,PS2}] \label{thm.zero}
For any finite type surface $F$ and any non-zero complex number $\zeta$, the skein algebra $\KF$ does not have non-trivial zero divisors, i.e. if $xy=0$ in $\KF$, then $x=0$ or $y=0$.
\end{theorem}
For the case when $F$ has at least one punctures and is triangulable (see below), this fact follows from the existence of the quantum trace map \cite{BW0} which embeds $\KF$ into a quantum torus which does not have non-trivial zero-divisors. For all finite type surfaces the theorem was proved in \cite{PS2}.

\subsection{Ideal triangulations and parametrization of simple diagrams} 
An {\bf ideal triangle} is the result of removing three points from the boundary of an oriented disk.  The three open intervals that are the components of the boundary are the {\bf sides} of the ideal triangle. An {\bf ideal triangulation} of a finite type surface  is given by a collection $\{\Delta_i$\} of ideal triangles whose sides have been identified in pairs to obtain a quotient space $X$, along with a homeomorphism $h:X\rightarrow F$. For each triangle there is an inclusion map $\Delta_i\rightarrow X$.  If the inclusion map is not an embedding then there are two sides of $\Delta_i$ that are identified to each other and the triangle is {\bf folded}. The images of the sides of the ideal triangles in $F$ are called the {\bf edges of the triangulation}.
 A {\bf triangulable surface} is  a surface which has an ideal triangulation. It is known that a finite type surface is triangulable if and only if it has at least one puncture and it is not  the sphere with one or two punctures. A {\bf triangulated finite type surface} is a finite type surface equipped with an ideal triangulation.
 If  $F$ is triangulable, then  it also has  triangulations with no folded triangles, and we  will only work with such triangulations.

Suppose $F$ is a triangulated surface. 
 Denote the set of edges by $E$. 
Then
$|E|=-3e$,  where $e$ is the Euler characteristic of $F$. 
A map from $E$ to $\BZ$ is called an {\bf edge-coloring}, and the set of all such maps is denoted by $\BZ^E$. An
edge-coloring $f$ is {\bf admissible} if $f \in \BN^E$, i.e. $f(a)\ge 0$ for  $a\in E$, and
whenever $a,b,c$ are three distinct edges  of an ideal triangle, 
\begin{itemize}
\item the sum $f(a)+f(b)+f(c)$ is even,
\item the integers $f(a),f(b),f(c)$  satisfy the triangle inequality, i.e. each is less than or equal to the sum of the other two.
\end{itemize}

\def\da{\boldsymbol{\delta}_a}

A simple diagram $\al$ on $F$ defines an admissible edge-coloring $f_\al \in \BZ^E$ by
$ f_\al(a) = i(\al,a).$
The map $\al \to f_\al$ is a bijection between the set $\sS$ of all isotopy classes of simple diagrams and the set of all admissible edge-colorings. For details, see e.g.~\cite{Matveev}.

While $\BZ^E$ is a $\BZ$-module, the subset of all admissible edge-colorings is not a $\BZ$-submodule of $\BZ^E$, because the values of  an admissible edge-coloring are non-negative.

 \begin{prop} \label{r.ev1} Suppose $F$ is a triangulated finite type surface, with $E$ the set of all edges. Then
 the $\BZ$-span of admissible edge-colorings contains $2\, \BZ^E$.

 \end{prop}

 \begin{proof} For an edge $a\in E$ let $\da\in \BZ^E$ be the edge-coloring defined by $\da(a)=1$ and $\da(b)=0$ for all $b \neq a$. We need to show that $2\da$ is a $\BZ$-linear combination of admissible edge-colorings.

 Recall that $F = \bF \setminus \cV$, and the ideal triangulation of $F$ has the set of vertices   $\cV$.
 Suppose first $a$ has two distinct endpoints $v_1, v_2\in \cV$. Let $\al$ be the boundary of a small tubular neighborhood of $a$ in $\hF$, and  let $\al_i$ be the boundary of a small disk in $\hF$ containing $v_i$, for $i=1,2$. Then it is easy to check that $2\da= f_{\al_1} + f_{\al_2} -f_{\al}$.

 Now suppose the two ends of $a$ meet at one point $v$. Let $\al$ be the boundary of a small tubular neighborhood of $a$ in $\hF$, and $\al'$ be the boundary of a small disk in $\hF$ containing $v$. Then $2\da=f_{\al'}-f_{\al}$.    \end{proof}
 
Instead of the edge coordinates $f_\al$ sometimes it is convenient to use the corner coordinates $g_\al$ described as follows. A {\bf corner} of a triangulation of $F$ consists of a triangle $T$ of the triangulation and an unordered pair $\{a,b\}$ of distinct edges of $T$. For a simple diagram $\al$ the corner number $g_\al(T, \{a,b\})$ is defined by
\begin{equation}
g_\al(T,\{a,b\})         =\frac{f_\al(a)+f_\al(b)-f_\al(c)}{2},
\label{corner}
\end{equation}
where $c$ is the remaining edge of $T$. The  geometric meaning of the corner number is as follows. After an isotopy we can assume 
  $\al$ is in  {\bf normal position}, that is,  $|\al \cap a|=i(a,\al)$ for all $a\in E$. 
 Then $\al$ intersects $T$ in a collection of arcs, each has its endpoints lying in two distinct sides of the triangle. The number of such arcs with endpoints in $a$ and $b$ is the corner number $g_\al(T,\{a,b\})$, see Figure \ref{normtri}.

\begin{figure}[H]\begin{center}\scalebox{0.7}{\begin{picture}(100,60)\includegraphics{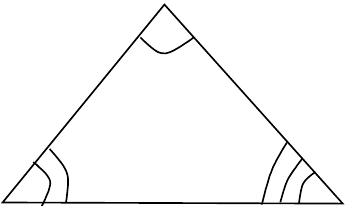}\put(-90,30){$a$}\put(-52,-10){$b$}\put(-20,30){$c$}\end{picture}}\end{center}\caption{Diagram in Normal Position, with $g_\al(T,\{a,b\})=2$.}\label{normtri}\end{figure}

The admissibility of $f_\al$ translates to the condition: for every edge $e$,
\be \label{eq.corner1}
g_\al(c_3) - g_\al(c_1) = g_\al(c_2)- g_\al(c_4),
\ee 
 where $c_1, c_2, c_3, c_4$  are the corners adjacent to $e$ depicted in Figure \ref{fig:Q}. Denote the value of both sides of  \eqref{eq.corner1} by $q_\al(e)$,
 \be \label{eq.q}
 q_\al(e)= g_\al(c_3) - g_\al(c_1)= (f_\al(a)- f_\al(b) + f_\al(c) - f_\al(d))/2,
 \ee
 where $a,b,c,d$ are the edges of the two triangles having $e$ as an edge, as depicted in Figure \ref{fig:Q}. Note that there might be some coincidences among $a,b,c,d$.
\FIGc{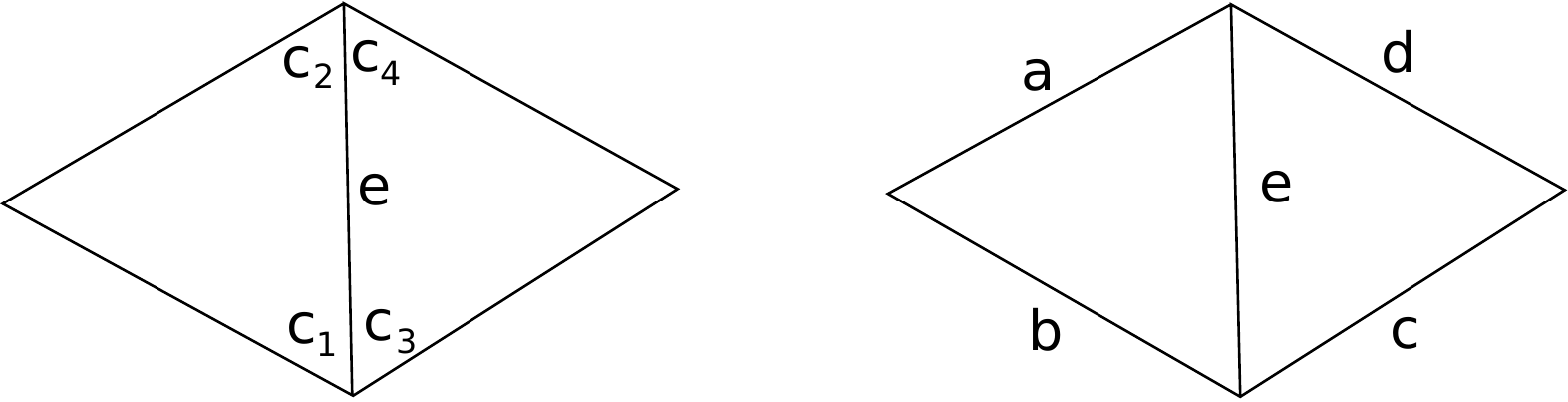}{Edge $e$ with adjacent corners (left) and surrounding edges (right).}{2cm}
 
 Let $C$ be the set of corners of the ideal triangulation of $F$. An {\bf admissible corner-coloring}  is a function
$ g:C\rightarrow \mathbb{N}$ satisfying the condition \eqref{eq.corner1} (with $g$ in place of $g_\al$). 

 The map $f_\al \to g_\al$ is   
an  $\BN$-linear bijection between admissible corner-colorings  and admissible edge-colorings of the edges. For details, see e.g. \cite{Matveev}.

\def\ld{\mathrm{lt}}

The set $\sS$ of isotopy classes of simple diagrams can be ordered as follows. 
\begin{definition}\label{wellorder}
Choose a total order on the set of edges $E$, then order $\BZ^E$ lexicographically. This induces a total order on the set $\sS$ via the bijection $\al\to f_\al$. If $\al=\sum_{j\in J}c_j \al_j$ is the standard presentation \eqref{eq.pres} of a non-zero $\al\in \KF$, then the {\bf lead term} $\ld(\alpha)$ of $\alpha$ is $ c_i \al_i$, where $\al_i$ is the largest among all $\al_j$.
\end{definition}

It was proved in \cite{AF2} that the lead terms behave well with respect to multiplication.

 \begin{theorem}[\cite{AF2}]\label{lead}  Suppose $F$ is a triangulated finite type surface and $\zeta$ is a non-zero complex number.
 
 (a) If $x,x'$ are non-zero elements in $\KF$, then
\be \label{eq.lead}
\ld(x\,  x')= \ld \big( \ld (x) \, \ld (x')\big),
\ee

(b) If $\al,\beta\in \sS$ then there is $k\in \BZ$
 such that 
 $
 \ld(\al\beta)= \zeta^{k} [f_\al + f_\beta]$, where  and $[f]$ is the simple diagram whose edge-coloring is $f$.

  \end{theorem}

The number $k$ in part (b) can be calculated explicitly, as in the following statement.

\begin{prop} \label{r.com} 
Suppose $\al,\beta\in \sS$. Let $q_\al\cdot f_\beta$ be the dot product. One has
\begin{align} \label{eq.6t}
\ld(\al \beta) &=  \zeta^{q_\al \cdot f_\beta}\, [f_\al + f_\beta]
\\
\label{eq.6s}
\ld(\al\beta) & = \zeta^{2(q_\al\cdot f_\beta)}\, \ld(\beta\al),\\
i_2(\al,\beta)&= q_\al\cdot f_\beta \pmod 2.
\label{eq.itwo}
\end{align}

\end{prop}
\begin{proof} Identity \eqref{eq.6t} can be proved easily  by taking the lead part of the Bonahon-Wong quantum trace map \cite{BW0,Le:qt}. Here we present an alternative, elementary proof. 

\FIGc{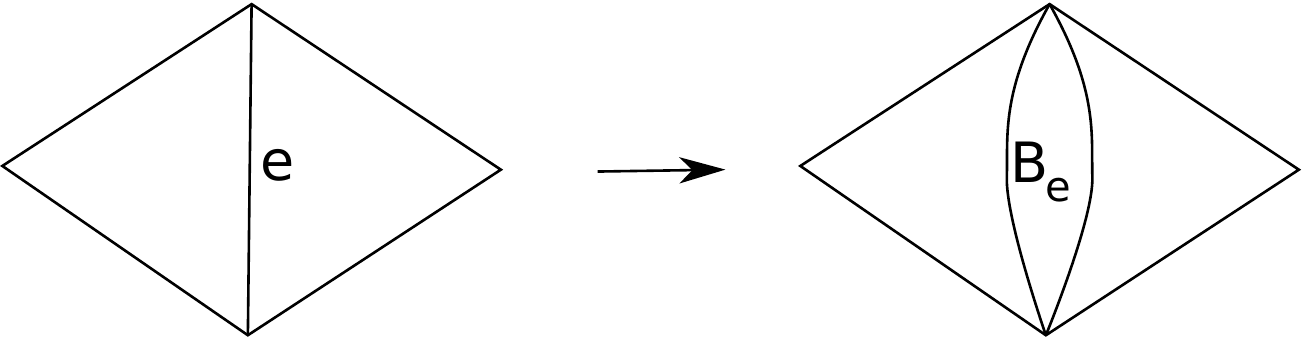}{Turn edge $e$ into bigon $B_e$. Do this to all edges.}{2cm}

Enlarge each edge $e$ to make a bigon $B_e$, see Figure \ref{fig:bigon1}. Now $F$ is the union of ideal triangles and bigons. Isotope $\al$ and $\beta$ so that in each triangle the arcs of $\al$ are closer to the vertices than the arcs of $\beta$ and hence they do not intersect. This means $\al$ and $\beta$ intersect only in the bigons. Let us look at a bigon $B_e$, see Figure \ref{fig:bigon2}, where we assume that $e$ is drawn vertical.

\FIGc{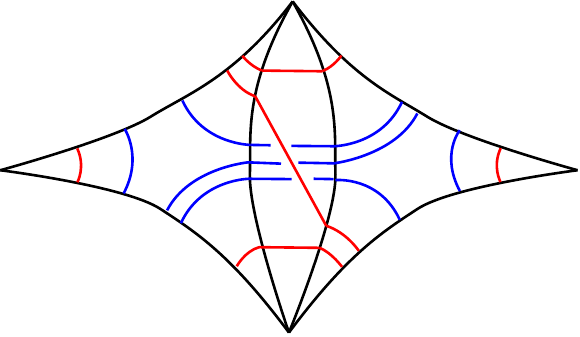}{$\al$ is  red, $\beta$ is blue.  Each line represents a set of parallel arcs.}{3cm}

 In Figure \ref{fig:bigon2}, arcs of $\al$ are  red and arcs of $\beta$ are  blue. The arcs of $\beta$ in $B_e$ are all horizontal, and there are $f_\beta(e)$ of them. The arcs of $\al$ in $B_e$ have 3 groups: the top and the bottom ones are horizontal, and the middle ones are slanted. It is easy to show that there are $|q_\al(e)|$ number of slanted arcs. Only slanted arcs intersects the arcs of $\beta$. The slanted arcs have positive or negative slope according as $q_\al(e) <0$ or $q_\al(e) >0$. If $q_\al(e)=0$ then there are no slanted arcs. In $B_e$, arcs of $\al$ intersects arcs of $\beta$ in $f_\beta(e) |q_\al(e)|$ number of points, as every slanted arc intersects every arc of $\beta$. It follows that modulo 2, the number of intersection points is $q_\al \cdot f_\beta$, proving \eqref{eq.itwo}.
 
 Now resolve the crossing of the product $\al \beta$, presented as $\al\cup \beta$ with $\al$ above $\beta$. To get the maximal diagram all the resolutions must be positive or all must be negative according as $q_\al(e) >0$ or $q_\al(e) < 0$. It follows that
\be 
\ld(\al \beta) = \zeta^{q_\al\cdot f_\beta} [f_\al + f_\beta].
\ee
Now calculate the product $\beta \al$ using $\beta \cup \al$ with $\beta$ above $\al$, to get
\be 
\ld(\beta\al) = \zeta^{-q_\al\cdot f_\beta} [f_\al + f_\beta].
\ee
From the above two identities we get \eqref{eq.6s}.
\end{proof}
\def\bd{\boldsymbol{\delta}} 

\begin{remark} If $(\sigma_{ab})_{a,b\in E}$ is the anti-symmetric matrix defined  in \cite{BW0} (which is also
 the face matrix $Q$ in \cite{Le:qt}), then   $q_\al (a)= \frac 12 \sum_{b\in E} \sigma_{ab} f_\al(b)$. Hence $q_\al\! \cdot\! f_\beta=-\frac 12 \sum_{a, b} \sigma_{ab} f_\al(a) f_\beta(b)$ is an anti-symmetric bilinear form on $\BZ^E$. Identity \eqref{eq.6t} shows that the algebra of lead terms maps into the quantum torus defined by the anti-symmetric form $\sigma$, which is a version of the Chekhov-Fock algebra. The quantum trace map is an algebra map extending this one to all lower terms, see details in \cite{BW0,Le:qt}.
\end{remark}

\subsection{Dehn-Thurston coordinates}\label{DehnTh} Suppose $F$ is a closed oriented surface of genus $g>1$. Since $F$ does not have an ideal triangulation, edge-coordinates cannot be defined. Instead the Dehn-Thurston coordinates are usually used. There are several version of the Dehn-Thurston coordinates. Here we  follow the presentation of \cite{LS}.

Recall that a {\bf pants decomposition} of a surface $F$ is a collection  $\cP$ of disjoint simple closed curves on $F$ such that every component of the complement of their union is a thrice punctured sphere (i.e. a pair of pants). Any pants decomposition of $F$ consists of  $3g-3$ curves.  

We encode the data for Dehn-Thurston coordinates with a  pair $(\cP,\cD)$,
where  $\cP=\{P_i\}_{i=1}^{3g-3}$ is a pants decomposition of $F$, and $\cD$ is an embedded  {\bf dual graph}. The graph $\cD$ has a trivalent vertex in every pair of pants, and  for each simple closed curve $P_i\in \cP$  has an edge intersecting  it transversely in a single point that is disjoint from the remaining curves in $\cP$. Both ends of that edge might be at the same vertex if the closure of the pair of pants is a surface of genus one.
This description is equivalent to the hexagonal decomposition described in  \cite{LS}.

For each curve $P_i\in \cP$ choose an  annulus $A_i\subset F$ having $P_i$ as its core, in such a way that the annuli $A_i$ are disjoint from one another, and that $\cD\cap A_i$ is a single arc.  The closure of each connected component  $c$ of $F -  \cup A_i$   is a pair of pants $S_c$. 
We think of the $S_c$ as {\bf shrunken pairs of pants}, hence the notation. Notice that the boundary components of $S_c$ coincide with the boundary components of some of the annuli. We number the boundary components  $c_1,c_2,c_3$ so that the edges of the dual graph  that intersect $c_i$ are ordered counterclockwise around the vertex. In diagrams we always number the boundary components of a shrunken pair of pants so that $c_1$ is the outer boundary component.

\begin{definition}\label{omega}
Given a pants decomposition  and an embedding of its dual graph $(\cP ,\cD)$ let $\Omega$ be the simple diagram that is isotopic to the boundary of a regular neighborhood of $\cD$.
Figure \ref{omegar} shows   the curves in the pants decomposition, the dual graph, and the diagram $\Omega$ for a surface of genus $2$. \end{definition} 

\begin{figure}
\begin{center}
\scalebox{0.67}{\includegraphics{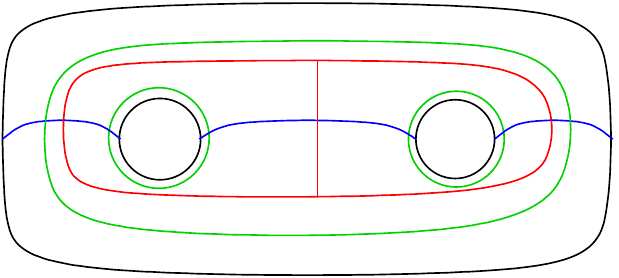}}
\end{center}
\caption{ $\cP$ in blue, $\cD$ in red, $\Omega$ in green}\label{omegar}
\end{figure}

We note that the simple diagram $\Omega$ has geometric intersection number $2$ with each curve in $\cP$.
It intersects each annulus $A_i$ in two arcs that are parallel to the arc of intersection of $\cD$ with $A_i$. The intersection of $\Omega$  with any  shrunken pair of pants $S_c$ consists of three properly embedded connected  $1$-manifolds,  each of which has its endpoints in distinct boundary components of $S_c$.  If $\partial S_c =c_1\cup c_2\cup c_3$ we name these arcs $d_1,d_2,d_3$, where the index of $d_i$ indicates which of the curves $c_i$ is disjoint from $d_i$.
We call the curves $d_i$ the {\bf triangular model curves}, and we picture them in Figure \ref{modeldual}.

\begin{figure}\begin{center}\scalebox{0.6}{\begin{picture}(115,115) \includegraphics{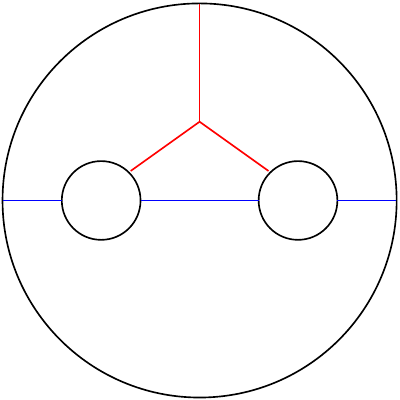}\put(-10,10){$c_1$} \put(-20,40){$c_3$} \put(-80,40){$c_2$}\end{picture}} \end{center}\caption{ $\cD\cap S_c$ in red, model curves $d_i$ in blue}\label{modeldual}\end{figure}
We also need {\bf nontriangular model curves} that have both endpoints in the same boundary component of $S_c$,  missing $\cD\cap\partial S_c$.
For each $c_i$ there are two isotopy classes  of connected properly embedded essential $1$-- manifolds that have both ends in $c_i-\cD$ and intersect $\cD\cap S_c$ transversely in a single point in one of the edges.  We choose  $u_i$ to be the curve that intersects $\cD$ in the edge ending in $c_{(i+1)\mod 3}$, see Figure \ref{sameends}.

\begin{figure}\begin{center}\scalebox{0.60}{\begin{picture}(115,115) \includegraphics{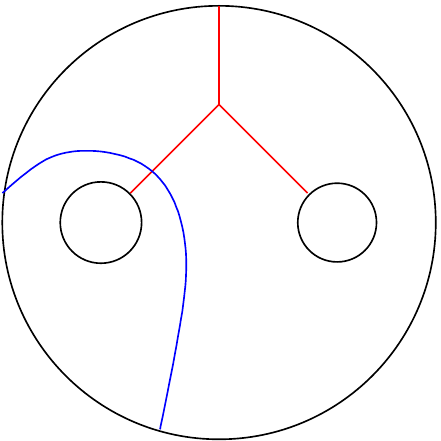}\put(-10,10){$c_1$} \put(-20,40){$c_3$} \put(-80,40){$c_2$}\end{picture}} \end{center}\caption{ Model curve $u_1$ joining the first boundary component to itself}\label{sameends}\end{figure}

\begin{figure}
\begin{center}
\scalebox{0.67}{\includegraphics{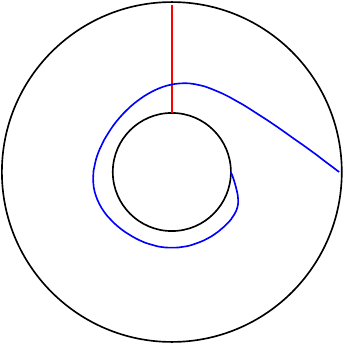}}
\end{center}
\caption{Positive twist}\label{postwist}
\end{figure}

Any simple diagram $S$ on $F$ can be isotoped so that it does not form a bigon with the boundary of any $A_i$, and inside each shrunken pair of pants  $S_c$ the components of $S\cap S_c$ are parallel to the model curves $d_i$ and/or $u_i$, by an isotopy that does not pass through the intersection of the dual graph $\cD$ with the boundaries of the $S_c$'s. We call this {\bf standard position}.

\begin{definition}
The {\bf Dehn-Thurston coordinates} of a simple diagram $S$ on a surface $F$ with respect to a pants decomposition and an embedding of its dual graph $(\cP,\cD)$ are given by 
 \begin{equation}(n_1(S),\dots ,n_{3g-3}(S),t_1(S),\dots ,t_{3g-3}(S)),\end{equation}
 as follows.  Assume that $S$ is in standard position with respect to  $(\cP,\cD)$. The first $3g-3$ coordinates are given by the geometric intersection numbers of the diagram with the curves in 
    $\cP$, i.e.
  \begin{equation}
  n_j(S)=i(S,P_j).
  \end{equation}
  We call these the {\bf pants coordinates} of $S$.
  The following $3g-3$ coordinates are the twisting  numbers of $S$ about the annuli $A_j$, that is 
  \begin{equation}
  |t_j(S)|=i(S\cap A_j, \cD\cap A_j).
\end{equation}
The sign of $t_j(S)$ is given by the following convention: If $S$ is in standard position, the twist number $t_j(S)$ is positive if as you travel along a component of $S\cap A_j$ that intersects $\cD\cap A_j$, passing from one boundary component of  $A_j$ towards the other, you are turning to the right.
In Figure \ref{postwist} the blue arc twists once in the positive direction. The $(t_1(S),\dots ,t_{3g-3}(S))$ are called the {\bf twist coordinates} of $S$.
\end{definition}

We can detect the presence of nontriangular model curves from the pants coordinates.  We say $P_i,P_j,P_k$ {\bf bound a shrunken pair of pants} if they are isotopic to the boundary components of some $S_c$. Notice that two of the curves could coincide if the shrunken pair of pants corresponds to a subsurface of genus $1$.  
\begin{definition}
The tuple  $\bn=(n_1,\dots,n_{3g-3})\in \BN^{3g-3}$  is {\bf  triangular} if for every  $(P_i,P_j,P_k)$ that bounds a shrunken pair of pants,  the sum $n_i+ n_j+ n_k$ is even and each of $n_i, n_j, n_k$ is less than or equal to the sum of the other two.  A simple diagram $S$ is
  {\bf triangular} if its pants coordinates are triangular, and no curves in $S$ are parallel to a curve in $\cP$. 
  
  \end{definition}
 It is worth noting that when a triangular simple diagram is in standard position then it contains no nontriangular model curves.

 \begin{theorem}\label{twistedpants}Given a finite collection of simple diagrams $\mathcal{S}$ on a closed surface $F$ of negative Euler characteristic there exists  a pants decomposition $\cP$ of $F$, and an embedding of the dual graph $\cD$ such that all the diagrams in ${\mathcal{S}}$ are triangular with respect to $(\cP,\cD)$.\end{theorem}

 \begin{proof}

 Let $\Rzero$ be the set of all non-negative  real numbers. For each positive integer $k$ the  projective space $P(\Rzero^k)$ is the quotient space $(\Rzero^k \setminus \{0\})/\sim$, where $\bx \sim \bx'$ if there is a $\lambda >0$ such that $\bx ' = \lambda \bx$.
 
 A tuple $\bx=(x_1,\dots,x_{3g-3})\in \Rzero^{3g-3}$  is {\bf strictly  positively  triangular} if each $x_i$ is positive and  for every  $(P_i,P_j,P_k)$ that bounds a shrunken pair of pants,  each of $x_i, x_j, x_k$ is less  the sum of the other two. 
 An element $\by \in P(\Rzero^{3g-3})$ is {\bf strictly positively  triangular} if one, and hence all, of its representatives in $\Rzero^{3g-3}$ are strictly positively triangular.
 The set of strictly positively triangular elements of $P(\Rzero^{3g-3})$ is open and non-empty.

Let $\mathcal{J}$ be the set of all isotopy classes of simple closed curves on $F$ and let $\mathcal{J}_{\mathcal{S}}$ be the subset of $\mathcal{J}$ coming from components of diagrams in $\mathcal{S}$. Let $(\cP',\cD')$ be a pants decompostion and an embedding of its dual graph   in $F$, giving rise to Dehn-Thurston coordinates. The {\bf projective pants coordinate map} on $\mathcal{J}$ is the composition of the maps
\be 
\pc:\mathcal{J} \to \BN^{3g-3} \embed \Rzero^{3g-3} \to P(\Rzero^{3g-3}),
\ee
where the first is the  pants coordinate map given by $J \to (n_j(J))_{j=1}^{3g-3}$ for any $J\in \mathcal{J}$. If $\pc(J)$ is strictly positively triangular, then clearly $J$ is triangular.

 The set $\mathcal{J}$  embeds as a dense subset of the set $\PMF$ of all projective measured foliations. We refer the reader to \cite{FLP} for  the definition and standard facts about measured foliations.
 The projectve pants coordinate map  can be continuously extended to 
  \be 
 \pc : \PMF \to P(\Rzero^{3g-3}).
 \ee
 Hence the inverse image of the set of strictly positively triangular elements of  $P(\Rzero^{3g-3})$ is open in $\PMF$.
 
The set $\PMF$ is the quotient of the set $\MF$ of all Whitehead-equivalence classes of measured foliations  by the action of multiplicative group  of all positive real numbers. The natural action of the mapping class group on $\PMF$ is minimal in the sense that every non-empty subset of $\mathcal{PMF}$ that is invariant under the mapping class
group is dense, see \cite{MP} or \cite[Theorem 6.19]{FLP}. It follows that  the set of (projective classes) of stable measured foliations of all pseudo-Anosov
homeomorphisms is dense in $\mathcal{PMF}$.  Therefore there exists a pseudo-Anosov $\phi: F \rightarrow F$
whose stable foliation $\sigma_{\phi}$ is in the inverse image of strictly  positively triangular elements with respect to the Dehn-Thurston coordinates from 
$(\cP',\cD')$.
A fundamental property of pseudo-Anosov homeomorphisms (see for instance \cite{FLP})
is that, for every simple closed curve $J$, $\phi^n(J)$ converges to the stable foliation $\sigma_{\phi}$ in 
 $\mathcal{PMF}$ as $n\rightarrow \infty$. In particular,  for $n$ large enough, $\phi^n(J)$ is positively triangular with respect to
 $(\cP',\cD')$.  Applying this to each $J\in \mathcal{J}_{\mathcal{S}}$ we see that, for $n$ large enough and for any diagram $S\in \mathcal{S}$,
$\phi^n(S)$  is positively triangular and hence triangular.
Hence, if we choose $n$ large enough, every diagram $S \in \mathcal{S}$  is triangular with respect to Dehn-Thurston coordinates from 
 $(\cP,\cD)=(\phi^{-n}(\cP'), \phi^{-n}(\cD'))$. 
\end{proof} 
The authors thank the referee for suggesting this shorter proof.

\def\ep{\epsilon}
\section{Characterization of the center}\label{center}Fix a finite type surface  $F=\bF \setminus \cV$ and a root of unity $\zeta$ of order $n$.  In this section we  characterize the center $\ZF$ of the skein algebra $\KF$ .
  
   Let $m$ be the order of $\zeta^4$,  $m'$ be the order of $\zeta^2$, and $\epsilon= \zeta^{m^2}$, i.e.
  \begin{equation} \label{mepsilon}
 m=\frac{n}{gcd(n,4)}, \quad m'=\frac{n}{gcd(n,2)},\quad  {\mbox{and}}\ \epsilon= \zeta^{m^2}. \end{equation}
 Note that  $\epsilon \in \{\pm 1,\pm {\bf i}\}$. More precisely
 \be 
 \label{eq.4cases}
  \epsilon = \begin{cases} \pm 1 \quad &\text{if} \ n \neq 4 \pmod 8\\
 \pm {\mathbf i} \quad &\text{if} \ n = 4 \pmod 8.
 \end{cases}
 \ee
Recall that a peripheral loop in $F$ is the boundary of any closed disk in $\bF$ containing exactly one point of $\cV$. 
If $F$ is closed then it has no peripheral loops. Clearly any skein corresponding to a peripheral loop is central in $\KF$, for any non-zero complex number $\zeta$.
 
When $\zeta$ is a root of $1$ of order $n$, there is another type of central element,   described in Theorem \ref{imageofCH}. 
  We will prove that $\ZF$ is the sub-algebra of $\KF$ generated by these two types of central elements. The proof is different for a surface with at least one puncture than for a closed surface. In the first case we use a triangulation  of the surface to parametrize simple diagrams, and in the second case we use Dehn-Thurston coordinates. In both cases we make use of a graded algebra associated to a filtration of the skein algebra.
  
\subsection{Formulation of results} For a subalgebra $A\subset \KF$ let $A[\partial]$ be the subalgebra of $\KF$ generated by $A$ and the peripheral loops.
 
  \def\la{\langle}
 \def\ra{\rangle}
 \def\CX{\BC\la X \ra_\epsilon}
 \def\Ch{Ch}
 \def\CKd{Ch(K_{\epsilon}(F))[\partial]}
 \def\XX{\mathrm{XX}}
 \def\CXd{\Ch(\CX)[\partial]}

\begin{theorem}
\label{thm.center0} Suppose $F$ is a finite type surface and $\zeta$ is a root of 1 of order $n$. The center $\ZF$ of the skein algebra $\KF$ is 
\be \label{eq.center}
\ZF= \begin{cases} Ch(\KeF)[\partial]\quad &\text{if } n \neq 0 \mod 4\\
 Ch(K^\ev_\ep(F))[\partial]  &\text{if } n = 0 \mod 4.
\end{cases}
\ee

 \end{theorem}

 A particular case is when $\zeta^4=1$, or $\zeta \in \{\pm 1, \bf \pm i\}$. In this case $m=1$, 
 $\ep=\zeta$,  and $Ch: \KF\to \KF$ is the identity map. Hence, Theorem \ref{thm.center0}  becomes
 \begin{prop}  One has
 $$
\ZF=\begin{cases} \KF \quad &\text{if } \ \zeta=\pm 1 \\
K^\ev_{\zeta}(F) &\text{if }\ \zeta = \pm \bf i.
\end{cases}
$$
 \end{prop}

 \begin{remark} If $\zeta$ is not a root of 1 then $\ZF$ is the subalgebra generated by the peripheral elements, see \cite{PS2}. 
 \end{remark}
 
 To combine the two cases of theorem \ref{thm.center0} into one, let us define 
\begin{equation}X= \begin{cases}\sS  \quad &   \text{ if } n\neq 0 \pmod 4\\
\sS^\ev  &   \text{ if } n= 0 \pmod 4.
\end{cases}
\end{equation}   Let $\CX$ be the $\BC$-span of $X$ in $\KeF$,
then the right hand side of \eqref{eq.center} is $\CXd$. From Theorem \ref{imageofCH} and the fact that peripheral loops are central, we already have
\be  \label{eq.in5}
 \CXd  \subset \ZF.
 \ee
 To prove Theorem \ref{thm.center0} we need to prove the reverse inclusion of \eqref{eq.in5}.

 \def\lt{\mathrm{lt}}
 \def\dr{\partial^{\bf r}}
 \subsection{Surfaces with at least one puncture}\label{sec.opensurface}
  In this section we prove Theorem \ref{thm.center0} for the case when $F$ has at least one puncture. If $F$ is a sphere with less than 3 punctures, then $\KF$ is commutative, and Theorem \ref{thm.center0} is true  since peripheral elements generate $\KF$ in this case. Hence, we will assume that $F$ is not a sphere with less than 3 punctures. In this case $F$ is triangulable, and we fix an ideal triangulation of $F$. Let $\partial_1,\dots, \partial_p$ be the peripheral skeins. For $\br=(r_1,\dots,r_p)\in \BN^p$ let $\dr= (\partial_1)^{r_1} \dots (\partial_p)^{r_p}$.

 Fix a total order of the set  of edges. Define the lead term of a non-zero element of $\KF$, and the functions  $f_\al, g_\al$, and $q_\al$ for $\al\in \sS$ as in Section \ref{review}.
\begin{lemma} \label{r.r2}
Suppose $0\neq z\in \ZF$. After rescaling $z$ by a nonzero complex number, there exists $\al\in X$ and $\br\in \BN^p$ such that $\ld(z) = \ld(Ch(\al) \dr)$.
 \end{lemma}

\begin{proof} After rescaling we can assume that $\al':=\ld(z)$ is a simple diagram.
Suppose $\beta$ is another simple diagram. From $z\beta = \beta z$ and \eqref{eq.lead} we get
\be 
\label{eq.r6} \ld(\al'\beta)= \ld(\beta\al').
\ee
By \eqref{eq.6s} one has $\ld(\al'\beta)= \zeta^{2(q_{\al'}\cdot f_\beta)}\,\ld(\beta\al')$. It follows that $ \zeta^{2(q_{\al'}\cdot f_\beta)}=1$; or $q_{\al'}\cdot f_\beta\in m'\BZ$ since $m'=\ord(\zeta^2)$. Thus,
\be 
\label{eq.r6g} q_{\al'}\cdot f  \in m'\BZ
\ee
holds for $f=f_\beta$ for all simple diagrams $\beta$. Hence 
\eqref{eq.r6g} is true for all $\BZ$-linear combination of $f_\beta$. 
By Proposition \ref{r.ev1}, \eqref{eq.r6g} holds for $2\bd_e$ for any edge $e$. This means $2 q_{\al'}(e)\in m'\BZ$, or equivalently $q_{\al'}(e)\in m\BZ$ for all edges $e$.

By taking out all the peripheral loops in $\al'$ we get a simple diagram $\al''$. Hence 
$\al'= \al'' \dr$ for some $\br=(r_1,\dots,r_p)\in \BN^p$. Note that $q_{\al''}=q_{\al'}$. 
Since $\al''$ does not have peripheral loops, at every vertex $v\in \cV$ there is a corner $c$ such that the corner number $g_{\al''}(c)=0$. If $c,c'$ are two adjacent corners at $v$ sharing  a common edge $e$, then  \eqref{eq.q} shows that $|g_{\al''}(c)- g_{\al''}(c')|= |q_{\al''}(e)|\in m\BZ$. By starting at a corner with 0 value of $g_{\al''}$, we  see that $g_{\al''}(c)\in m\BZ$ for all corners $c$. It follows that $\al''= \al^m$, where $\al$ is the simple diagram with $g_\al = g_{\al''}/m$. We have $\ld(z)= \al^m \dr$.

Since $Ch(\al)$ is the product of $T_m(\al_j)$ for all components $\al_j$ of $\al$, and $\dr$ can be isotoped away from any diagram,  we have $\al^m \dr= \ld(Ch(\al) \dr)$. Thus $\ld(z)= \ld(Ch(\al) \dr)$.

It remains to show that $\al\in X$. If $n\neq 0 \pmod 4$ then $X=\sS$ and there is nothing to prove. Suppose $n=0 \pmod 4$, then $m'=2m$, and $X= \sS^\ev$. Since $q_{\al'}= q_{\al''}= m q_\al$, from \eqref{eq.r6g} 
we get $q_\al \cdot f_\beta\in 2\BZ$, for all simple diagrams $\beta$. By \eqref{eq.itwo}, we have $i_2(\al, \beta)= q_\al \cdot f_\beta \pmod 2 =0$. This shows $\al \in \sS^\ev$.
\end{proof}

 \begin{proof}[Proof of Theorem \ref{thm.center0}] Let $0\neq z \in \ZF$. By Lemma \ref{r.r2}, after rescaling  $z$ by a non-zero complex number there are $\al\in X$ and $\br\in \BN^p$ such that $\ld(z) = \ld(Ch(\al) \dr)$.  By Theorem \ref{imageofCH}, $Ch(\al)$ is central. Thus $Ch(\al) \dr$ is also central. Hence $z'= z - Ch(\al) \dr$ is a central element with smaller lead term. By induction  $z\in Ch(\CX)[\partial]$, proving the theorem.
 \end{proof}

\subsection{Closed surfaces}\label{centerclosed}

 Now we consider the case where $F$ is a closed surface, which is much more complicated than the open surface case. The proof of Theorem \ref{thm.center0} for this case will occupy the rest of Section \ref{center}.
 
Since $F$ is closed, there are no peripheral elements, and  Theorem \ref{thm.center0}  says that 
\be 
\ZF= Ch(\CX).
\ee

If $F$ is a sphere then the skein algebra is commutative and Theorem \ref{thm.center0} holds. When $F$ has genus 1, Theorem \ref{thm.center0} holds, it is an easy consequence of the product to sum formula \cite{FG}. In \cite {AF1}  the case $n= 2 \pmod 4$ is handled explicitly for genus one surfaces. Hence , we  assume that $F$ has genus greater than $1$.

The outline of the proof is as follows. We choose a  pants decomposition of the closed surface $F$ and use Dehn-Thurston coordinates to parametrize
 simple diagrams on $F$ as described in section \ref{DehnTh}. We  filter the algebra $K_\zeta(F)$ using the pants coordinates. This gives an associated graded algebra. We consider a subspace of the graded object generated by  triangular diagrams. We show that if a central skein  is a linear combination of triangular diagrams  then  each of these triangular diagrams is an $m$-th power of another triangular diagram in $X$.  Finally, given a skein in the center of the algebra,  by Theorem  \ref{twistedpants} we can choose the coordinates so that  the skein is a linear combination of triangular diagrams, and  proceed by induction where we subtract  elements of  $Ch(\CX) $ to get a simpler skein.

\subsection{Filtrations and gradings}We start by recalling some general facts about filtrations and associated gradings of algebras.
Suppose $\cB$ is a  $\BC$-vector space
 graded by an abelian monoid $I$. This means that $\cB= \bigoplus_{i\in I} \cB_i$, where each $\cB_i$ is a subspace of $\cB$.
A subspace $M\leq \cB$ is said to {\bf respect the grading} if
\begin{equation}M =\bigoplus_{i\in I} (M\cap \cB_i).\end{equation}

\begin{lemma} \label{r.decomp}
Suppose $\cB$ is a $\BC$-algebra, and the grading is
 {\em compatible} with the algebra structure, i.e. $\cB_i \cB_j \subset \cB_{i+j}$, then the  center $Z(\cB)$ of $\cB$ respects the grading.
 \end{lemma}
 \begin{proof}
 Let $z$ be a central element, with  $z=\sum z_i$, where $z_i\in \cB_i$. 
If $x\in \cB_k$ then
\begin{equation} xz= \sum _i  x z_i, \quad zx = \sum _i z_ix.\end{equation}
 Note that both $z_ix$ and $xz_i$ are in $\cB_{i+k}$. Comparing elements having the same gradings, we get $z_ix = xz_i$, which shows each $z_i$ is central.
 \end{proof}

Let $\cA$ be a $\BC$-algebra.  Suppose $\{ F_n\}$ is a {\bf filtration} of $\cA$, where the index set is $\BN$.  This means that each $F_n$ is a $\BC$-submodule of $\cA$,
\begin{equation}
F_{n} \subset F_{m} \ {\mbox{if}}\  n<m ,\  \  F_n F_k \subset F_{n+k},{\mbox{  and}}\ \cup_{n\in \BN}F_n=\cA.
\end{equation}
One defines the associated graded algebra  $\Gr\cA$ as follows.
If  $x\in \cA$  then  $x\in F_n$ for some $n$, and the smallest such $n$ is called the 
 {\bf degree} of $x$, denoted by
   $\gr(x)$. One has $\gr(xy) \le \gr(x) + \gr(y)$.
Let
$\Gr_n \cA = F_n /F_{n-1}$ and 
\begin{equation} \Gr\cA =\bigoplus \Gr_n\cA.\end{equation} 
Let $\tau_n: F_n \to \Gr_n\cA$ be the natural projection.
 Define the product on $\Gr\cA$ as follows:
\begin{equation}
 {\mbox{If  }}\ \ 
 \gr(x)=k,\  \gr (y)=l  \ \ {\mbox{  then  }}\ \ 
\tau_k(x) \tau_l(y) = \tau_{k+l}(xy).
\end{equation}
Then $\Gr\cA$ is a graded algebra. 
Define $\ld : \cA \to \Gr\cA$ by 
\begin{equation}\label{tau}
\ld (x)=\tau_n(x) \ \ 
{\mbox{if}} \ \  \gr(x)=n.
\end{equation}
Note that $\ld $ does not respect the multiplication. However, we
have the following.

\begin{equation}\label{eq.com1}
\ld (x) \ld  (y)  =
\begin{cases}   \ld (xy)
 \quad &\text{if } \gr(xy)= \gr(x) + \gr(y) \\
    0      &\text{if }  \gr(xy)< \gr(x) + \gr(y).
\end{cases}
\end{equation}
Moreover,

\begin{equation}\label{eq.gr}
\gr(x) = \gr(\ld (x)).
\end{equation}
In particular, if $x\neq 0$ then $\ld (x) \neq 0$.
From \eqref{eq.com1} it is easy to show
\begin{lemma}   \label{r.sub}
If $\cA$ is a $\BC$-algebra with filtration indexed by $\BN$, $\Gr\cA$ is the associated graded algebra and the map 
$\ld $ is defined by (\ref{tau}) then 
\begin{equation}\ld (Z(\cA)) \subset  Z(\Gr \cA).\end{equation} 
\end{lemma}

\subsection{Filtration of the Kauffman bracket skein algebra by pants coordinates}

 Fix a pair of pants decomposition $\cP=\{P_i\}_{i=1}^{3g-3}$ and a dual graph $\cD$ in $F$ so that one can define the Dehn-Thurston coordinates $\bn(\al),\bt(\al)$, for  $\al\in \sS$, as in section~\ref{DehnTh},  where $\bn(\al)\in \BN^{3g-3}$ and $\bt(\al)\in \BZ^{3g-3}$.

 \begin{definition}\label{Ind}
 Let $\I\subset \BN^{3g-3} \times \BZ^{3g-3}$  be the subset  consisting of 
 $(\bn,\bt)$  satisfying: 
  \begin{itemize}
 \item If $(i,j,k)$ bound a pair of pants, then $n_i+ n_j + n_k$ is even;
 \item   If $n_i=0$, then $t_i \ge 0$.
 \end{itemize}
 \end{definition}
 There is a bijection between elements of $\I$ and $\sS$ via Dehn-Thurston coordinates. Denote a simple diagram corresponding to $(\bn,\bt)$ by $S(\bn,\bt)$.

\def\Fd{F^\diamond}

 Recall that an element $\bn\in \BN^{3g-3}$  is   {\bf triangular} if for every  $(i,j,k)$ that bounds a pair of pants, $n_i+ n_j + n_k$ is even and each of $n_i, n_j, n_k$ is less than or equal to the sum of the other two. 
 Let $\di\in \BZ^{3g-3}$ be the vector whose coordinates are 0, except for the $i$-th one, which is 1.  Although $\di$ is not triangular, we will show that $2\di$ is a $\BZ$-linear combination of triangular tuples. 
  \begin{lemma}\label{r.even} 
    The $\BZ$-linear span of all triangular elements in  $\BN^{3g-3}$ contains $2\BZ^{3g-3}$.
    \end{lemma} 
    
\begin{proof} The surface $\Fd:=F \setminus \cD$ is a finite type surface.  Let $a_j= P_j\cap F'$.   The elements of the collection $\{a_j\}$   are the edges of an ideal  triangulation  of $\Fd$. 
 Simple diagrams $\al$ on $\Fd$ are parameterized by admissible edge-colorings $f_\al:\{ a_i\} \to \BN$, where $f_\al(a_j)= i(\al,a_j)$. Note that $f:\{a_j\}\to \BN$ is admissible if and only if $(n_j=f(a_j))_{i=1}^{3g-3}$ is triangular. Hence the lemma follows from Proposition \ref{r.ev1}.
  \end{proof}

 For any simple diagram $S(\bn,\bt)$ define the $\BN$-grading of  $S(\bn,\bt)$ to be
\begin{equation}\label{graddiag}
\gr(S(\bn,\bt))= |\bn|=\sum _i n_i.
\end{equation}

\def\cF{\mathcal F} 
For $n\in \BN$,  let $\cF_n \subset \KF$  be the $\BC$-subspace spanned by all $\al\in \sS$  with $\gr(\al) \le n$.  
The filtration $\{ \cF_n \}_{n\in \BN}$ is  compatible with the algebra structure of $\KF$.
Denote the  associated graded algebra by $\cB$, 
$ \cB = \bigoplus F_n/F_{n-1},$
with a bijective map $\ld: \KF \to \cB$. As $\sS$ is a $\BC$-basis of $\KF$, the set $\lt(\sS)$ is a $\BC$-basis of $\cB$, and we often identify $\sS$ with $\lt(\sS)$ when we say that the set $\sS$ of simple diagrams is a $\BC$-basis of $\cB$.

\def\bq{\mathbf q}

\subsection{Triangular subspaces}
 Recall that a simple diagram $\al=S(\bn,\bt)$ is  triangular if $\bn$ is triangular and  $t_i=0$ whenever $n_i=0$.
Let $\IT\subset \I$ be the subset of all $(\bn,\bt)\in \I$ such that $S(\bn,\bt)$ is triangular.
Let  $\Bt$ to be the $\BC$-subspace of $\cB$ spanned by all triangular simple diagrams.
 This subspace was first considered in \cite{PS2} where it is proved that 
  $\Bt$ is a subalgebra of $\cB$, and that $\cB$ is a domain. 
  
 The finite type surface $\Fd= F \setminus \cD$ has an ideal triangulation with edges $\{ a_j\}$, where $a_j = P_j \setminus \cD$. Suppose $\bn\in \BN^{3g-3}$ is triangular. The function $\{ a_j\} \to \BN$, given by $a_j \to n_j$, is an admissible coloring, and  defines a simple diagram $\gamma\subset \Fd$. Define $\bq(\bn)\in \BN^{3g-3}$ so that its $j$th-coordinate is $q_{\gamma}(a_j)$, where $q_\gamma$ is defined by \eqref{eq.q}.

\def\sgn{\mathrm{sgn}}
\def\cF{\mathcal F}

   \begin{prop}\label{r.PS2} Suppose 
    $(\bn,\bt), (\bn',\bt')\in \IT$.  Let $\bq'= \bq(\bn')$, then
   \be 
   \label{eq.com1u}
   S(\bn,\bt)  \, S(\bn',\bt') = \zeta^{ - \bn\cdot \bq' + \bt \cdot \bn' - \bn\cdot \bt'} \, S(\bn+\bn', \bt+\bt') \quad \text{ in $\cB$}.
   \ee
   Consequently, 
 \be  \label{eq.com1v}
   S(\bn,\bt)  \, S(\bn',\bt') = \zeta^{2( - \bn\cdot \bq' + \bt \cdot \bn' - \bn\cdot \bt')} \, S(\bn',\bt')  \, S(\bn,\bt) \quad \text{ in $\cB$}.
   \ee
   \end{prop}
   \begin{proof} Let $S^1$ be the standard circle with a fixed base point $*\in S^1$. The point $*$ together with its antipodal point cut $S^1$ into two open half-circles denoted by $H$ and $H'$. 
   Recall that the annulus $A_j$ is a small regular neighborhood of $P_j$.
  Choose an identification $A_j \equiv S^1 \times [-1,1]$, 
  such that $A_j \cap \cD \equiv \{*\} \times [-1,1]$, and equip
 $A_j$ with the  product Riemannian metric of $S^1 \times [-1,1]$. Choose a simple diagram $\al\in \Fd \cup (\cup_j A_j)$ representing $S(\bn,\bt)$ such that (in each $A_j$) one has $\al \cap \partial A_j = U_j \times \{-1, 1\}$, where $U_j\subset H$ and $|U_j|= n_j$,  and each  connected component of $\al \cap A_j$ is a geodesic. Similarly,  choose a simple diagram $\al'\in \Fd \cup (\cup_j A_j)$ representing $S(\bn',\bt')$ such that (in each $A_j$) one has $\al' \cap \partial A_j =U'_j \times \{-1, 1\}$, where $U'_j\subset H$ and $|U'_j|= n'_j$, and each connected  component of $\al' \cap A_j$ is a geodesic.

 In each ideal triangle $T$, which is a connected component of $\Fd \setminus (\cup_j\mathring A_j)$, where $\mathring A_j$ is the interior, we can assume that each connected component of $\al \cap T$ intersects each connected component of $\al'\cap T$ in at most one point. Let $\al \bullet \al$ be the link diagram which is $\al \cup \al'$ with $\al$ above $\al'$. At every double point $x$ of $\al \bullet\al'$ there is only one resolution, call it the good resolution, which does not produce a diagram of lesser degree. Let $\sgn(x)$ be the sign (in the skein relation)
  of this resolution. Denote $s_1= \sum \sgn (x)$ where the sum is over all double points in the annuli, and $s_2= \sum \sgn (x)$ where the sum is over all double points outside the annuli. 
Let $(\al\bullet \al')^{\#}$ be the result of doing good resolution at every double point. 
 Then 
 \be \label{eq.com9}
 \al \al' = \zeta^{s_1 + s_2  } (\al\bullet \al')^{\#} \quad \mod \cF_{|\bn| + |\bn'|}(\KF).
 \ee
 From  \cite[Lemma 17]{PS2}, one has $s_1= \bt \cdot \bn' - \bn\cdot \bt'$ 
 and $(\al\bullet \al')^{\#} = S(\bn+\bn', \bt + \bt')$. 
 
  Let  $\al_0=\al$ outside the annuli while inside an annulus $A_j$, $\al_0$ is the union of $n_j$ vertical lines $U_j \times [-1,1]$. Then $\al_0$ represents $S(\bn,0)$. Similarly, let $\al'_0=\al'$ outside the annuli and   $\al'\cap A_j =U'_j \times [-1,1]$. Then \eqref{eq.com9} for $\al_0, \al'_0$ becomes
  \be \label{eq.com9a}
 \al_0 \al'_0 = \zeta^{s_2  } S(\bn+ \bn',0) \quad \mod \cF_{|\bn| + |\bn'|}(\KF).
 \ee
  Identity \eqref{eq.6t} shows that the factor $\zeta^{s_2  }$ is equal to $\zeta^{\bq(\bn) \cdot \bn'}$. Since this is true for all non-zero $\zeta$, we have $s_2= \bq(\bn) \cdot \bn' = -\bn \cdot \bq'$. 
  Hence $s_1 + s_2= - \bn \cdot \bq' + \bt \cdot \bn' - \bn\cdot \bt'$, and \eqref{eq.com9} proves the proposition.
 \end{proof}

Let $X^\Delta$ be the set of all triangular simple diagrams in $X$. Let $Z(\Bt)$ be the center of the algebra $\Bt$. Recall that $m = n/gcd(n,4)$ is the order of $\zeta^4$.
\begin{lemma}\label{r.even7}
 Suppose $\al\in \sS$ and $\al^m\in Z(\Bt)$. Then $\al\in X^\Delta$.
\end{lemma}
\begin{proof} If $n\neq0 \pmod 4$ then $X=\sS$ and there is nothing to prove. Assume $n= 0 \pmod 4$, or $X= \sS^\ev$.

Assume to the contrary that $\al\not \in \sS^\ev$. Then there is $\beta\in \sS$ such that $i_2(\al,\beta)=1$. As  triangular simple diagrams span the homology group $H_1(F;\BZ)$, we can assume $\beta$ triangular. 
Since  both $\al,\beta$ are non-zero basis elements of $\cB$, the products $\al^m \beta, \beta \al^m$ are non-zero (by Proposition \ref{r.PS2} (a)).
By Equation \eqref{Cha} of Theorem \ref{imageofCH},
\be 
\ Ch(\al) \, \beta = - \beta\, Ch(\al).
\label{eq.k8}
\ee
Since $\al^m =\ld(Ch(\al))$, $\ld(\beta)=\beta$, and $\al^m \beta\neq 0$, from \eqref{eq.com1} we have $\al^m \beta= \ld(Ch(\al) \beta)$. Similarly $\beta \al^m = \ld(\beta Ch(\al) )$. Thus, taking $\ld$ of \eqref{eq.k8} we get
\be  
\label{eq.k9}\al^m \beta = - \beta \, \al^m.\ee
Both sides of \eqref{eq.k9} are non-zero, and \eqref{eq.k9} contradicts the fact that $\al^m$ is central in $\Bt$. We conclude that $\al \in \sS^\ev=X$. As $\al^m$ is triangular, $\al$ is also triangular. Thus, $\al\in X^\Delta$.
\end{proof}

\def\bode{\boldsymbol \delta}

\begin{prop} \label{r.c1} Suppose $(\bn,\bt)\in \IT$ and $\al=S(\bn,\bt) \in Z(\Bt)$.
Then $\al=\beta^m$ for some $\beta\in X^\Delta$.

\end{prop} 

\begin{proof}
Suppose $S(\bn,\bt) \in Z(\cB)$ and  $(\bn',\bt')\in \IT$. Recall that $m'=ord(\zeta^2)$ and define $f(\bn', \bt') $ to be half of the exponent of $\zeta$ in   Equation \eqref{eq.com1v},
\be
f(\bn', \bt')  =  -\bn \cdot \bq' + \bt\cdot \bn' - \bn \cdot \bt' \in m' \BZ.
\label{eq.new2}
\ee
 Note that if $(\bn', \bt')\in \IT$, then
$(\bn', \bt' + \bt'')\in \IT$ for all $\bt'' \in (\BZ_{>0})^{3g-3}$ whose entries are big enough (so that  $\bt' + \bt''$ has positive entries). Choose $\bt'' = km'(1,1,\dots,1) + \bode_i$ with $k$ big.
From $f(\bn', \bt') - f(\bn', \bt'+ \bt'') \in m' \BZ$ we see that $n_i \in m' \BZ$. Consequently $\bn \cdot \bq', \bn \cdot \bt' \in m'\BZ$. 
From \eqref{eq.new2}, we conclude that for all triangular $\bn'\in \BN^{3g-3}$,
\be  \label{eq.new3}
\bt\cdot \bn' \in m'\BZ.
\ee
Since $2 \bode_i$ is a $\BZ$-linear combination of some triangular  $\bn'\in\BN^{3g-3}$ (by Lemma \ref{r.even}), we also have $\bt\cdot \bode_i \in m'\BZ$. It follows that $2t_i \in m' \BZ$, or equivalently $t_i \in m\BZ$.

Hence $(\bn/m, \bt/m)$ has integer components.
It remains to show that for every $(i, j, k)$ that bounds a
pair of pants, $\frac{n_i}{m}+ \frac{n_j}{m} + \frac{n_k}{m}$ is even. In fact, if
$n \neq 0 \pmod 4$ then $m$ is odd, and since
$n_i + n_j + n_k$ is even, the sum $\frac{n_i}{m}+ \frac{n_j}{m} + \frac{n_k}{m}$ is even.
If $n = 0 \pmod 4$ then $m'=2m$, and
$\frac{n_i}{m}+ \frac{n_j}{m} + \frac{n_k}{m}\in \frac{m'}{m}\BZ = 2 \BZ$.

It follows that  $(\bn/m, \bt/m)$   is in $\IT$. Consequently, $\beta= S(\bn/m, \bt/m)$ is triangular, and $S(\bn,
\bt)= \beta^m$. By Lemma \ref{r.even7}, we have $\beta \in X$. Hence $\beta\in X^\Delta$.
\end{proof}

We are ready to complete the proof of Theorem \ref{thm.center0} for the case of the closed surface.
\begin{proof}
 Suppose $z$ is a non-zero central element of $\KF$, with standard presentation $z = \sum_{i\in I} c_i \al_i$  in the basis $\sS$.  By Theorem \ref{twistedpants} we can choose a pants decomposition $\cP$ of $F$, and an embedding of the dual graph $\cD$ such that each $\al_i$ is triangular. Let $\CXz$ be the $\BC$-linear subspace of $\KF$ spanned by the associated $X^\Delta$.
 Then $z\in \CXz$. We will prove by induction that if $z\in \CXz$ and is central, then $z\in Ch(\CX)$.
 Since  $\ld(z) \in \Bt$, its standard presentation is of the form 
 \be \label{eq.j1}
 \ld(z) = \sum_{j\in J} c_j S(\bn_j, \bt_j), \quad (\bn_j,\bt_j) \in \IT, \ \text{all } |\bn_j| \ \text{are equal}.
 \ee
 Since $\ld(z)\in Z(\cB)$ by Lemma \ref{r.sub} and  $\ld(z)\in \Bt$, we have $\ld(z) \in Z(\cB) \cap \Bt \subset Z(\Bt)$.
 
 For $(\bn,\bt)\in \IT$ let $\Bt_{(\bn,\bt)}$ be the $\BC$-subspace spanned by $S(\bn,\bt)$.  Then 
$$\Bt = \bigoplus_{(\bn,\bt)\in \IT} \Bt_{(\bn,\bt)},$$
 and \eqref{eq.com1u} shows that this a grading compatible with the algebra structure of $\Bt$. It then follows from Lemma \ref{r.decomp} that $S(\bn_j, \bt_j)\in Z(\Bt)$ for each $j\in J$.  By  
  Proposition  \ref{r.c1}, there is $\beta_j\in X^\Delta$ such that $S(\bn_j,\bt_j)= (\beta_j)^m$.

Note that if $\beta\in X^\Delta$, then any simple diagram obtained from $\beta$ by replacing a component of $\beta$ by several of its parallels is also in $X^\Delta$. As $Ch$ is defined by applying the polynomial $T_m$ to each component, we see that $Ch(\beta_j)\in \CXz$. Besides, $\ld(Ch(\beta_j))= (\beta_j)^m$,  because $T_m$ is a monic polynomial of degree $m$.

Let $z'= \sum_{j\in J} c_j Ch(\beta_j)$. Then $z'$ is central because it is an element in  $Ch(\CX)$. Besides, $z'\in \CXz$. Then $z-z'$ is also central, lying in $\CXz$, and having the grading less than that of $z$.  By induction, we conclude that $z\in Ch(\CXe) \subset Ch(\CX)$. This completes the proof of the theorem.
\end{proof}

\section{Finiteness}\label{finite}

In this section we prove that the skein algebra at a root of 1 is finitely generated as a module over its center and describe the algebraic set $\MS(\ZF)$. Throughout $F$ is a finite type surface and $\zeta$ is a root of 1 of order $n$, with $m=\ord(\zeta^4)$ and $\ep=\zeta^{m^2}$.
\subsection{Finite generation over center}
\begin{theorem} \label{finiteness} If $F$ is a finite type surface and $\zeta$  a root of unity then
$K_{\zeta}(F)$ is finitely generated as a module over its center. The number of generators is less than or equal to $(2m)^s$, where $m=\ord(\zeta^4)$ and  $s$ is the number appearing in Theorem \ref{spanning}. \end{theorem}

\begin{proof} . 
Let $J_1,\dots,J_s$ be the simple loops  of  Theorem \ref{spanning}. For $\bk=(k_1,\dots,k_s)\in \BN^s$ let
$$ J(\bk)= T_{k_1}(J_1) \dots T_{k_s}(J_s).$$
By Theorem \ref{spanning},  the set $ \{ J(\bk) \mid \bk \in \BN^s\}$ spans $K_\zeta(F)$ over $\BC$. Let $V\subset K_\zeta(F)$ be the $\ZF$-submodule spanned by the set
 $\{ J(\bk) \mid k_i < 2m\}$, which has $(2m)^s$ elements. We  prove that $V= K_\zeta(F)$ by showing that $J(\bk)\in V$ for all $\bk\in \BN^s$ using induction on $|\bk|= k_1+ \dots + k_s$.

 For any $q\in \BN$, $\zeta^{2mq}=1$, hence $T_{2mq}(J)\in \ZF$ for any simple closed curve $J$ (see \cite[Cor 2.3]{Le}).
 Suppose $k_j\ge 2m$ for some $j$.
 Dividing by $2m$, we get $k_j= 2mq + r$. By Equation \eqref{prodsum},
 \begin{equation} \label{fiveone}  T_{k_j}(J_j) = T_{2mq}(J_j) T_{r}(J_j) - T_{2m q-r} (J_j).\end{equation}
 Substituting  the right hand side of Equation (\ref{fiveone}) for $T_{k_j}(J_j)$ in the product $J(\bk)= T_{k_1}(J_1) \dots T_{k_s}(J_s)$,  we see that $J(\bk)$ is the sum of two terms, each of which is in $V$ by the induction hypothesis. Hence $J(\bk)\in V$, completing the proof. \end{proof}

\subsection{Variety of classical shadows} Since $\KF$ is finitely generated as a module over $\ZF$ and is affine over $\BC$, the Artin-Tate lemma (Lemma \ref{artintate}) shows that $\ZF$ is affine over $\BC$. Denote $\MS(\ZF)$ by $\cY_\zeta(F)$. Since $\ZF$ is an integral domain, $\cY_\zeta(F)$ is an algebraic variety. 
We will call it the {\bf variety of classical shadows} of $\KF$. Its coordinate ring $\BC[\cY_\zeta(F)]$ is $\ZF$.

A regular map $f: Y \to Y'$ between two affine algebraic varieties is {\bf finite} if the dual map $f^*: \BC[Y'] \to \BC[Y]$ is an embedding and $\BC[Y]$ is integral over $f^*(\BC[Y'])$. In this case, the degree of  $f$ is the rank of $\BC[Y]$  over $f^*(\BC[Y'])$. Any finite regular map is surjective, and the preimage of a point consists of no more than $d$ points, where $d$ is the degree. 
 
 Suppose $G$ is a finite group acting algebraically on an affine variety $Y$, i.e.  for every $g\in G$, the map $Y\to Y$ given by $y \to g\cdot y$ is a regular map. Then the quotient set $Y/G$ is naturally  an affine variety, and the quotient map $Y \to Y/G$ is a finite regular map of degree $\le |G|$, see \cite[Example 1, Section 5.3]{Sha}.
 
 For any element $\al\in \pi_1(F)$ there is a complex valued function on the set of representations 
$\rho: \pi_1(F) \to SL_2(\BC)$ which sends 
 $\rho$ to  the trace of the matrix $\rho(\al)$. 
The set $\cX(F)$ of all such trace functions on the $SL_2(\BC)$-representations of $\pi_1(F)$  is an affine variety. It is called the $SL_2(\BC)$-character variety of $F$ and its  dimension is
given by 

\be \label{eq.dim}
\dim \cX(F)= \begin{cases} 0 \quad &\text{if  } F= S^2 \\
2 &\text{if  } F= T^2 \\
 6g- 6+3p  &\text{if $S$ has $p\ge 0$ punctures},
\end{cases}
\ee
(see e.g.  \cite[Proposition 49]{Sikora}).

 The cohomology group $H^1(F;\BZ_2)$ is identified with  the set of all group homomorphisms $u: \pi_1(F) \to \{ \pm 1\}$. 
If $\rho: \pi_1(F) \to SL_2(\BC)$ is a representation and $u\in H^1(F;\BZ_2)$, then let  $u*\rho$ be the new representation defined by $(u*\rho)(\al)= u(\al) \rho(\al)$ for all  $\al\in \pi_1(F)$.
 The map $\rho \to u *\rho$ 
 descends to an algebraic action of $H^1(F;\BZ_2)$ on the $SL_2(\BC)$-character variety $\cX(F)$.
  Let $\cX'(F)$ be the quotient variety $\cX(F) /H^1(F;\BZ_2)$. 
When $F$ has at least one puncture, $\cX'(F)$ is the $PSL_2(\BC)$-character variety of $F$, and when $F$ is a closed surface,  $\cX'(F)$ is the connected component of the $PSL_2(\BC)$-characters containing the character of the trivial representation, see \cite{HP}.

Recall that $F= \bF\setminus \cV$, where $\bF$ is a closed surface and $\cV$ is a finite set. Via restriction, the cohomology group $H^1(\bF;\BZ_2)$ can be considered as a subgroup of $H^1(F;\BZ_2)$ and hence acts on $\cX(F)$. Let $\cX_0(F)$ be the quotient variety $\cX(F) /H^1(\bF;\BZ_2)$. We have the  following finite regular maps
\be 
\cX(F) \onto \cX_0 (F) \onto \cX'(F),
\ee
where the first one has degree $2^{b_1(\bF)}$ and the second one has degree $2^{b_1(F) - b_1(\bF)}$. The three varieties $\cX(F), \cX_0(F), \cX'(F)$ have the same dimension given by \eqref{eq.dim}.
When $\cV$ is empty or consists of 1 point, then $b_1(F) = b_1(\bF)$, and $\cX_0(F)= \cX'(F)$.

\begin{theorem} \label{zfaffine} Suppose $F$ is a finite type surface with $p$ punctures,
$\zeta$ is a root of 1 of order $n$, and let $m =\ord(\zeta^4)$.

The variety of classical shadows $\cY_\zeta(F)$ has dimension equal to that of the $SL_2(\BC)$-character variety $\cX(F)$, which is given by \eqref{eq.dim}.
More precisely, we have the following:

(i) If $n\neq 0 \mod 4$ then there is a finite regular map $f: \cY_\zeta(F) \onto \cX(F)$ of degree $\le m^p$. In particular, if $p=0$ (i.e. $F$ is closed), then $\cY_\zeta(F)$ is isomorphic to $\cX(F)$ and hence does not depend on $\zeta$.

(ii) If $n= 0 \pmod 4$  then there is a finite regular map $f: \cY_\zeta(Y) \onto \cX_0(F)$ of degree $\le m^p$. In particular, if $p=0$ (i.e. $F$ is closed), then $\cY_\zeta(F)$ is isomorphic to $\cX_0(F)$ and hence does not depend on $\zeta$.
\end{theorem}

\proof

Let $\ep=\zeta^{m^2}$ and $\partial_1, \dots,\partial _p$ be the peripheral skeins.

(i) Since $n \neq 0 \pmod 4$, we have $\ep=\pm1$. By Theorem \ref{thm.center0}, one has $\ZF= Ch(\KeF)[\partial]$. 
Recall that $Ch(\KeF)$ is the algebra generated by $T_m(\al)$, for all simple loops $\al$. 
Since $T_m$ is a polynomial of degree $m$, the set ${\boldsymbol{\partial}}=\{ \ \partial_1^{k_1} \dots \partial _p^{k_p} \mid k_i   \le m-1\}$ spans $\ZF$ over $Ch(\KeF)$. Note that $|{\boldsymbol{\partial}}|= m^p$. 
The embedding $\KeF\cong Ch(\KeF) \embed Ch(\KeF)[\partial]$ gives rise to a finite regular map $f: \cY_\zeta(F) \onto \MS(\KeF)$ of degree $\le m^p$. 
Since $\ep =\pm 1$, we have $\MS(\KeF)\cong \cX(F)$, see \cite{B,PS}. This proves (i).

(ii) Suppose  $n= 0 \pmod 4$. We have $\ZF= Ch(K^\ev_\ep(F))[\partial]$. Since each 
peripheral element is even, we have $T_m(\partial_j)\in Ch(K^\ev_\ep(F))$ for each $\partial _j$. It follows that as a module over  $Ch(K^\ev_\ep(F))$, $\ZF$ can be generated by $\le m^p$ elements. Hence we have a finite regular map
$$ f: \cY_\zeta(F) \onto \MS(K^\ev_\ep(F)).$$
Statement (ii) follows from the following lemma.
\begin{lemma}  The variety $\MS(K^\ev_{\ep}(F))$ is canonically isomorphic to $\cX_0(F)$.
\end{lemma}
\begin{proof} Since $\ep\in \{\pm 1,\pm \bf i\}$, by  Theorem \ref{thm.Keven}, $K^\ev_\ep(F)$ is canonically isomorphic to $K^\ev_{-1}(F)$.

By the definition of finite quotient,  $\cX_0(F)= \MS(K_{-1}(F)^{H^1(\bF;\BZ_2)})$, where $K_{-1}(F)^{H^1(\bF;\BZ_2)}$ is the set of elements in $K_{-1}(F)$ fixed by $H^1(\bF;\BZ_2)$.
 
 Suppose $u: \pi_1(\bF) \to \{\pm 1\}$ is an element of $H^1(\bF;\BZ_2)$ and $0\neq\al \in K_{-1}(F)$. Using the basis $\sS$, we have  $\al=\sum_{j\in J} c_j \al_j$ where $J$ 
 is a non-empty finite set, $0\neq c_j \in \BC$, and $\al_j\in \sS$ are distinct. 
 We have $u * \al= \sum _{j\in J} c_j u(\al_j) \al_j$, and  identity $u*\al= \al$ means that
 $u(\al_j)=1$ for all $j$. Thus, $u*\al=\al$ for all $u\in H^1(\bF;\BZ_2)$ if and only if $u(\al_j)=1$ for all $u\in H^1(\bF;\BZ_2)$, which is equivalent to $\al_j$ being even. Hence, $K^\ev_{-1}(F)^{H^1(\bF;\BZ_2)} = K^\ev_{-1}(F)$, and $\cX_0(F)= \MS(K^\ev_{-1}(F))$. This completes the proof of the lemma, and the theorem.
 \end{proof}

 \begin{remark} Using the basis $\sS$, one can show that degrees in both (i) and (ii) of Theorem \ref{zfaffine} are exactly $m^p$. We will address this in a future work.
 \end{remark}
 
 \begin{remark} When $\ep=-1$, there is a canonical isomorphism between the $SL_2(\BC)$-character variety $\cX(F)$ and $\MS(K_{-1}(F))$, see \cite{B,PS}. Although $K_{-1}(F)$ and $K_{1}(F)$ are isomorphic as algebras, see \cite{Ba}, the isomorphism is not canonical. However, there is twisted version $\cX^{\mathrm{twist}}(F)$ of $\cX(F)$, which is canonically isomorphic to $\MS(K_1(F))$, see \cite{BW0,Thurston}. Therefore when $\ep=1$ we can replace $\cX(F)$ in Theorem \ref{zfaffine} by $\cX^{\mathrm{twist}}(F)$, and the finite morphism $f$ is canonical.
  \end{remark}

\section{Irreducible representations of the skein algebra and shadows}\label{irreps}

In this section we connect to the work of Bonahon and Wong \cite{BW2} by showing how central characters correspond to  classical shadows.

An irreducible representation $\rho: K_{\zeta}(F)\rightarrow M_d(\mathbb{C})$ defines a central character $\chi_\rho: \ZF \to \BC$.
 If $n =2 \pmod 4$ then $\ep=-1$, and 
 $\ZF=
Ch(K_{-1}(F))[\partial]$.
The homomorphism $\chi_{\rho}$ is determined by the pullback to $K_{-1}(F))$,
\begin{equation}
\label{eq.tchi}  \tilde{\chi}_{\rho}:K_{-1}(F)\rightarrow \mathbb{C}\end{equation} and by the list of the values $\chi_{\rho}(\partial_i)$ satisfying 
\be \label{eq.punc}
T_m(\chi_{\rho}(\partial_i))= \tilde \chi_\rho(\partial_i).
\ee

 Since $K_{-1}(F)$ is canonically isomorphic to the coordinate ring of the $SL_2(\BC)$-characters $\cX(F)$, the map $\tilde{\chi}_{\rho}$ determines an $SL_2(\BC)$-character. An $SL_2(\BC)$-character and the collection of values $\chi_{\rho}(\partial_i)$ satisfying \eqref{eq.punc} are what Bonahon and Wong call a {\bf classical shadow}.

   From our computation of the center of $K_{\zeta}(F)$, Bonahon and Wong's classical shadows of irreducible representations of $K_{\zeta}(F)$ are in  canonical one-to-one correspondence with our central characters. For a closed surface there are no peripheral skeins, and the classical shadow consists of an $SL_2(\BC)$-character only.

Similarly, for the other cases they considered, namely when $n$ is odd, Bonahon and Wong's classical shadows coincide with the central characters or with the points of the variety $\cY_\zeta(F)$.

When the order  of $\zeta$ is arbitrary, we define the classical shadow of an irreducible representation $\rho$ to be the point of the variety of classical shadows 
$\cY_\zeta(F)$ corresponding to the central character of $\rho$.

We now prove Theorem \ref{THM.2} stated in the introduction, which was conjectured by Bonahon and Wong in the case when the order of $\zeta$ is $2\pmod{4}$. 

\begin{theorem} [Unicity Theorem for skein algebras]  \label{THM.2a}
Let 
$F$ be a
 finite type surface and $\zeta$  
a 
 root of unity.  There  
is a Zariski open dense subset $U$ of the variety of classical shadows $\cY_\zeta(F)$ 
 such that  each point of $U$ is the classical shadow of a unique (up to equivalence) irreducible representation of $\KF$.   All irreducible representations with classical shadows in $U$ have the same dimension $N$ which is  equal to the  square root of the rank of           
$\KF$ over $\ZF$. If a classical shadow is not in $U$, then it has at most $r$ non-equivalent irreducible representations, and each has dimension $\le N$. Here $r$ is a constant depending on the surface $F$ and the root $\zeta$.

\end{theorem}

\proof By Theorem \ref{thm.zero},  $K_{\zeta}(F)$ does not have zero divisors and hence   it is prime. By \cite{Bullock} (see also Theorem \ref{spanning}), $\KF$  is affine over $\mathbb{C}$. By Theorem  \ref{finiteness}, $\KF$ is  generated as  module over $\ZF$ by a set of no more than $r=(2m)^s$ elements, where $m = \ord(\zeta^4)$ and $s$ is the number depending on $F$ and appearing in Theorem \ref{spanning}. Hence Theorem \ref{unicity} applies. \qed

From the dimension consideration we see that $N^2 \le r$. We don't know any other relation between $N$ and $r$.

\begin{remark} In \cite{FK1} it is shown that if $F$ is a surface of negative Euler characteristic $e(F)$ having $p>0$ punctures, and the order of $\zeta$ is $2 \mod{4}$ then the  rank of $K_{\zeta}(F)$ over $\ZF$  is $m^{-3e(f)-p}= m^{6g -6+ 2p}$, where $g$ í the genus.  The proof carries over to the cases where the order of $\zeta$ is odd to give the same result. This means the dimension of a generic irreducible representation, i.e. the number $N$ in Theorem \ref{THM.2a}, is $N= m^{3g-3+p}$. It should be noted that for closed surfaces (i.e. $p=0$) Bonahon and Wong \cite{BW2}  proved the inequality $N \le 3g -3$. In a future work we calculate the number $N$ for all other cases of roots of 1 and all type of surfaces, closed or punctured.\end{remark}

\end{document}